\documentclass[12pt,reqno]{amsart}
\usepackage{amssymb,amsmath,amsthm}
\usepackage{graphicx}
\usepackage{subcaption}
\usepackage{fullpage}
\usepackage{scrextend}
\usepackage{siunitx}
\usepackage{enumerate}
\usepackage{tikz,calc}
\usepackage{tikz-cd}
\usetikzlibrary{automata,positioning}
\usepackage{comment}
\usetikzlibrary{calc}
\usepackage{epsfig,graphicx}
\usepackage{listings}
\usepackage{enumerate}
\usepackage{float}
\usepackage{bbm}
\usepackage[colorlinks=true, pdfstartview=FitH, linkcolor=blue, citecolor=blue, urlcolor=blue]{hyperref}
\usepackage{tabstackengine}
\usepackage{listings}
\stackMath

\newtheorem{theorem}{Theorem}[section]
\newtheorem{lemma}[theorem]{Lemma}
\newtheorem{corollary}[theorem]{Corollary}
\newtheorem{proposition}[theorem]{Proposition}

\theoremstyle{definition}
\newtheorem{definition}[theorem]{Definition}

\newtheorem{problem}[theorem]{Problem}
\newtheorem{remark}[theorem]{Remark}
\newtheorem{example}[theorem]{Example}

\newcommand{\F}{\mathbb{F}}

\newcommand{\PGL}{\operatorname{PGL}}
\newcommand{\GL}{\operatorname{GL}}
\newcommand{\AG}{\operatorname{AG}}
\newcommand{\PG}{\operatorname{PG}}
\newcommand{\Tr}{\operatorname{Tr}}
\newcommand\qbin[3]{\left[\begin{matrix} #1 \\ #2 \end{matrix} \right]_{#3}}

\title{Extremal Peisert-type graphs without the strict-EKR property}

\author{Sergey Goryainov}
\address{School of Mathematical Sciences, Hebei International Joint Research Center for Mathematics and Interdisciplinary Science\\Hebei Normal University, Shijiazhuang  050024, P.R. China}
\email{sergey.goryainov3@gmail.com}
\author{Chi Hoi Yip}
\address{Department of Mathematics \\ University of British Columbia \\ Vancouver  V6T 1Z2 \\ Canada}
\email{kyleyip@math.ubc.ca}
\keywords{Erd\H{o}s-Ko-Rado theorem, Peisert-type graph, maximum clique, affine polar graph, Hilton-Milner theorem}
\subjclass[2020]{05C25, 51E15, 11T30, 15A03, 05C60, 30C10}

\begin{document}

\maketitle

\begin{abstract}
It is known that Paley graphs of square order have the strict-EKR property, that is, all maximum cliques are canonical cliques. Peisert-type graphs are natural generalizations of Paley graphs and some of them also have the strict-EKR property. Given a prime power $q \geq 3$, we study Peisert-type graphs of order $q^2$ without the strict-EKR property and with the minimum number of edges and we call such graphs extremal. We determine number of edges in extremal graphs for each value of $q$. If $q$ is a a square or a cube, we show the uniqueness of the extremal graph and classify all maximum cliques explicitly. Moreover, when $q$ is a square, we prove that there is no Hilton-Milner type result for the extremal graph, and show the tightness of the weight-distribution bound for both non-principal eigenvalues of this graph. 
\end{abstract}


\section{Introduction}

Throughout the paper, let $p$ be a prime and $q$ a power of $p$, and let $\F_q$ be the finite field with $q$ elements. We also let $\F_q^+$ and $\F_q^*=\F_q \setminus \{0\}$ be the additive group and multiplicative group of $\F_q$, respectively. 

Given an abelian group $G$ and a connection set $S \subset G \setminus \{0\}$ with $S=-S$, the {\em Cayley graph} $\operatorname{Cay}(G, S)$ is
the undirected graph whose vertices are elements of $G$, such that two vertices $g$ and $h$ are adjacent if and only if $g-h \in S$. 

We begin with the definition of the central objects in this paper.

\begin{definition}[Peisert-type graphs]\label{defn:peisert-type}
Let $q$ be a prime power. Let $S \subset \F_{q^2}^*$ be a union of $m \leq q$ cosets of $\F_q^*$ in $\F_{q^2}^*$, that is,
\begin{equation}\label{eq:coset}
S=c_1\F_q^* \cup c_2\F_q^* \cup \cdots \cup c_m \F_q^*.
\end{equation}
Then the Cayley graph $X=\operatorname{Cay}(\F_{q^2}^+, S)$ is said to be a Peisert-type graph of type $(m,q)$. 
\end{definition}

Let us point out that while Peisert-type graphs were introduced formally recently \cite{AGLY22, AY22}, they can date back to the construction of semi-primitive cyclotomic strongly regular graphs due to Brouwer, Wilson, and Xiang~\cite{BWX99} around 20 years ago. Peisert-type graphs can be also viewed as fusions of certain amorphic cyclotomic association schemes \cite[Section 5.4]{DM10}.
However, these connections are not necessary in this paper. Recall that for an odd prime power $q$, the {\em Paley graph of order $q^2$} is $\operatorname{Cay}(\F_{q^2}^+,(\F_{q^2}^*)^2)$, where $(\F_{q^2}^*)^2$ is the subset of squares in $\F_{q^2}^*$. The definition of Peisert-type graphs is clearly motivated by the well-studied Paley graphs and Peisert graphs, which are only defined in finite fields with odd characteristics (for example, it does not make sense to talk about squares in a field with characteristic $2$) \cite[Lemma 2.13]{AY22}. Thus, in \cite{AGLY22, AY22}, only Peisert-type graphs with odd characteristics were considered. However, the definition of Peisert-type graphs naturally extends to finite fields with characteristic 2, and we also consider Peisert-type graphs of even order in this paper. Note that the case $q = 2$ is trivial because for $m = 1$ and $m = 2$ there exists a unique Peisert-type graph, namely, the disjoint union of two edges and the 4-cycle, respectively. We thus assume that $q \ge 3$ in the following discussion.

Let $X$ be a Peisert-type graph of type $(m,q)$. It is well
known that $X$ is a strongly regular graph \cite[Corollary 5]{AGLY22} and thus the Delsarte-Hoffman bound implies that the clique number of $X$ is $q$ \cite[Theorem 7]{AGLY22}. From the decomposition of the connection set into $\F_q^*$-cosets in equation~\eqref{eq:coset}, it is clear that translates of $c_1\F_q, c_2\F_q, \ldots, c_m\F_q$ are (maximum) cliques in $X$. These cliques are known as the \emph{canonical cliques} in $X$ and we say $X$ has \emph{the strict-EKR property} (sometimes we also say that \emph{the EKR theorem holds for $X$}) if all maximum cliques in $X$ are canonical \cite{AGLY22, GM15, Y24}. For brevity, when we say a clique is \emph{non-canonical}, then the clique is implicitly assumed to be maximum. Such terminology is reminiscent of the classical Erd\H{o}s-Ko-Rado (EKR) theorem \cite{EKR61}, where all maximum families of $k$-element subsets of $\{1, 2, \ldots, n\}$ are canonically intersecting (that is, there is a common element in the intersection) whenever $n \geq 2k+1$. In \cite{AGLY22}, Asgarli, Goryainov, Lin, and Yip used the fact that Peisert-type graphs can be realized as block graphs of orthogonal arrays coming from point-line incidences on affine planes and showed that the definition of canonical cliques in the two families of graphs agree with each other. There are lots of combinatorial objects where an analogue of the EKR theorem holds; we refer to the book by Godsil and Meagher \cite{GM15} for a general discussion on EKR-type results. 

Blokhuis~\cite{B84} proved the EKR theorem for Paley graphs of square order.  Sziklai~\cite{S99} extended Blokhuis' result to a family of generalized Paley graphs. It is easy to verify that such generalized Paley graphs are Peisert-type graphs \cite[Lemma 2.13]{AY22}. Recently, Asgarli and Yip \cite{AY22} showed that the strict-EKR property holds for a family of Peisert-type graphs with nice algebraic properties on the connection set, generalizing the result by Blokhuis~\cite{B84} and Sziklai~\cite{S99} and confirming a conjecture of Mullin~\cite[Chapter 8]{M09} on Peisert graphs. Motivated by these extensions of the EKR theorem for Paley graphs, very recently, Yip \cite{Y24} showed that the strict-EKR property holds for almost all Peisert-type graphs of type $(\frac{q+1}{2},q)$ as $q \to \infty$. In other words, almost all pseudo-Paley graphs of Peisert-type have the strict-EKR property, indicating that Paley graphs are perhaps not too special among the family of pseudo-Paley graphs in the viewpoint of EKR. 

However, it remains widely open to completely classify Peisert-type graphs with the strict-EKR property \cite[Problem 29]{AGLY22}. Indeed, since Peisert-type graphs can be realized as block graphs of orthogonal arrays \cite{AGLY22}, this problem is an interesting sub-problem of a more general open problem on characterizing orthogonal arrays whose block graphs satisfy the strict-EKR property \cite[Problem 16.4.1]{GM15}. On the other hand, it is known that there are Peisert-type graphs of small order without the strict-EKR property \cite[Example 2.21]{AY22}. Moreover, there is an infinite family of Peisert-type graphs for which the strict-EKR property fails. This phenomenon was briefly discussed in \cite[Section 5]{AGLY22}, and in particular, the following result was proved:

\begin{theorem}[{\cite[Section 5]{AGLY22}}]\label{thm:OAEKR}
If $q>(m-1)^2$, then all Peisert-type graphs of type $(m,q)$ have the strict-EKR property. Moreover, when $q$ is a square, there exists a Peisert-type graph of type $(\sqrt{q}+1,q)$ that fails to have the strict-EKR property.
\end{theorem}

In \cite{Y24}, it is also shown that if $m$ is very close to $q$, then, perhaps unsurprisingly, almost all Peisert-type graphs of type $(m,q)$ fail to satisfy the strict-EKR property, as $q \to \infty$. In this paper, we are interested in determining the ``smallest" graphs without the strict-EKR property and exploring interesting algebraic properties in these graphs.

Let $X$ be a Peisert-type graph of type $(m,q)$ without strict EKR-property. We say that $X$ is \emph{strict-EKR extremal} (or simply \emph{extremal}), if all Peisert-type graphs of type $(m-1,q)$ have the strict-EKR property, in other words, $X$ is minimum in terms of failing the strict-EKR property. Theorem~\ref{thm:OAEKR} states there exists an extremal Peisert-type graph of type $(\sqrt{q}+1,q)$ when $q$ is a square. The main goal of the present paper is to study extremal graphs more closely and prove more refined structural results. We also classify large cliques (in particular, non-canonical cliques) in these graphs. 

Our first main result refines Theorem~\ref{thm:OAEKR} and determines the parameter of an extremal Peisert-type graph. Moreover, in our proof, for each prime power $q$, we construct an extremal Peisert-type graph of order $q^2$ explicitly.

\begin{theorem}\label{thm:thm1}
Let $q=p^n \geq 3$, where $p$ is a prime and $n$ is a positive integer. Let $X$ be an extremal Peisert-type graph defined over $\F_{q^2}$.
\begin{itemize}
    \item If $n=1$, then $X$ is of type $(\frac{p+3}{2},p)$ and $X$ is unique up to isomorphism.
    \item If $n>1$, then $X$ is of type $(p^{n-k}+1,q)$, where $k$ is the largest proper divisor of $n$.
\end{itemize}
\end{theorem}

Our second main result classifies extremal Peisert-type graphs when $q$ is square. Moreover, we explore the structure of cliques and eigenfunctions in these graphs. To formulate the next theorem, we need the notion of the weight-distribution bound, which is a lower bound for the size of the support of an eigenfunction corresponding to a non-principal eigenvalue of a distance-regular graph \cite[Corollary 1]{KMP16}; see Section~\ref{sec:WDB} for more background. 

\begin{theorem}\label{thm:thm2}
Let $q$ be the square of a prime power. There are exactly $(q+1)\sqrt{q}$ distinct extremal Peisert-type graph defined over $\F_{q^2}$, but they are all isomorphic. Each Peisert-type graph of type $(3,q)$ can be uniquely extended to an extremal graph. Moreover, if $X$ is such an extremal graph, then the following statements hold:
\begin{itemize}
    \item $X$ is isomorphic to the affine polar graph $VO^+(4,\sqrt{q})$.
    \item Maximum cliques in $X$ can be explicitly determined. In particular, $X$ has exactly $\sqrt{q}+1$ canonical cliques containing $0$, and $\sqrt{q}+1$ non-canonical cliques containing $0$; moreover, these $2(\sqrt{q}+1)$ cliques lie in the same orbit under the action of the automorphism group of $X$.
    \item There is no Hilton-Milner type result: all maximal cliques in $X$ are maximum cliques.
    \item The weight-distribution bound is tight for both non-principal eigenvalues.
\end{itemize}
\end{theorem}

In the spirit of EKR, analogues of the celebrated Hilton-Milner theorem \cite{HM67} have been studied extensively: these results concern the structure of maximal extremal configurations that are not maximum. For many Peisert-type graphs, we expect there are Hilton-Milner type results, that is, there is a maximal clique with size less than $q$, and the structure of the second largest maximal clique can be well predicted. We refer to  \cite{AY22, BEHW96, GKSV18, GMS22, GSY23, S99, Y22, Y23, Y24} for extensive discussions on maximal cliques as well as the stability of canonical cliques in generalized Paley graphs and Peisert-type graphs. However, the above result shows that there is no Hilton-Milner type result in an extremal Peisert-type graph of type $(\sqrt{q}+1,q)$, which is surprising. More surprisingly, canonical cliques and non-canonical cliques turn out to be equivalent under the action of the automorphism group of $X$. We remark that recently a result of a similar flavor has been proved in the setting of permutation groups, namely, there is no Hilton-Milner type result for the action on $\GL_2(\F_q)$ on $\F_q \times \F_q \setminus \{(0,0)\}$  \cite{MR21}. 

Our third main result classifies extremal Peisert-type graphs and the structure of maximum cliques therein when $q$ is a cube (but $q$ is not a square).

\begin{theorem}\label{thm:thm3}
Let $q=r^3$, where $r$ is a prime power and a non-square. There are exactly $r(r^5+r^4+r^3+r^2+r+1)$ distinct extremal Peisert-type graphs defined over $\F_{q^2}$, but they are all isomorphic. Moreover, if $X$ is such an extremal graph, then maximum cliques in $X$ can be explicitly determined; in particular, $X$ has exactly $r^2+1$ canonical cliques containing $0$, and $r^2+r+1$ non-canonical cliques containing $0$.
\end{theorem}

Our main results can be also formulated in terms of the structure of non-collinear point sets with size $q$ in $\AG(2,q)$ that determine the minimum number of directions; see Remark~\ref{rem:direction}. When $q=p^n$, where the smallest prime divisor of $n$ is at least $5$, contrary to Theorem~\ref{thm:thm2} and Theorem~\ref{thm:thm3}, extremal Peisert-type graphs defined over $\F_{q^2}$ may not be unique up to isomorphism. We refer to Example~\ref{ex:q=32} for two non-isomorphic extremal graphs when $q=2^5$. 

\textbf{Structure of the paper.}  
In Section~\ref{prelim}, we recall some basic terminologies and give some preliminary results. In Section~\ref{parameter}, we prove Theorem~\ref{thm:thm1}. The proof of Theorem~\ref{thm:thm2} and Theorem~\ref{thm:thm3} will be gradually developed in the subsequent sections. In Section~\ref{sec:count}, we analyze the structure of non-canonical cliques. In Section~\ref{sec:induced}, we construct extremal Peisert-type graphs induced by special geometric objects, and we will see their connection with affine polar graphs in Section~\ref{sec:polar}. The minimum size of the support of eigenfunctions will be discussed in Section~\ref{sec:eigenfunction}. Lastly, we discuss a question related to extremal block graphs of orthogonal arrays in Section~\ref{sec:OA}. The proof of Theorem~\ref{thm:thm3} can be found at the end of Section~\ref{sec:induced1}, and the proof of Theorem~\ref{thm:thm2} can be found at the end of Section~\ref{sec:eigenfunction}. We also include some computational results in Appendix \ref{sec:computation}.

\section{Preliminaries}\label{prelim}
In this section, we list some preliminary definitions and results.

\subsection{Affine and projective designs}
For the background on combinatorial designs, we refer to \cite{BJL99}.

For a set of points $V$, a set of blocks $B$, and an incidence relation $I$, the finite incidence structure $D = (V, B, I)$ is called a \emph{block 
design} with parameters $v, k, \lambda$ where $v, k, \lambda$ are positive integers if it satisfies the following  
conditions: 
\begin{enumerate}
    \item $|V| = v$;
    \item any two distinct points are joined by exactly $\lambda$ blocks. 
    \item any block is incident to exactly $k$ points.
\end{enumerate}
For a block design with parameters $v, k, \lambda$, the constants $r$ and $b$, where $r$ is the number of points incident to a given block and $b$ is the total number of blocks, are uniquely determined by $v,k$, and $\lambda$ (see \cite[Theorem 2.10, p. 10]{BJL99}).

Let $F$ be a field and $W$ an $n$-dimensional vector space over $F$. 
Then the set of all cosets of subspaces of $W$, ordered by inclusion, is called 
the \emph{$n$-dimensional affine space} over $F$. If $F=\F_q$, it 
will be denoted by $\AG(n,q)$. The cosets of $\{0\}$ are called \emph{points}, those of 
1-dimensional subspaces \emph{lines}, those of 2-dimensional subspaces \emph{planes}, those 
$(n-1)$-dimensional subspaces \emph{hyperplanes}, and in general the cosets of $i$-dimensional subspaces are called \emph{$i$-dimensional flats} or just \emph{$i$-flats}.  We also use the \emph{Gaussian coefficient} $\qbin{n}{i}{q}$ to denote the number 
of $i$-dimensional subspaces of an $n$-dimensional vector space over $\F_q$. The points 
of $\AG(n, q)$ together with the $d$-dimensional flats of $\AG(n, q)$ as blocks and  
incidence by natural containment form an incidence structure denoted by $\AG_d(n, q)$. 

\begin{proposition}[{\cite[Proposition 2.13, p. 11]{BJL99}}]
 $\AG_d(n, q)$ is a block design with parameters 
 $v = q^n$,
 $k = q^d$,
 $r = \qbin{n}{d}{q}$
 $\lambda = \qbin{n-1}{d-1}{q}$ 
 and 
 $b = q^{n-d}\qbin{n}{d}{q}$. 
\end{proposition}

Let $F$ be a field and $W$ an $(n + 1)$-dimensional vector space 
over $F$. Then the set of all subspaces of $W$, ordered by inclusion, is called the 
\emph{$n$-dimensional projective space} over $F$. If $F$ is the finite field $\F_q$, it will 
be denoted by $\PG(n, q)$. The 1-dimensional (2-dimensional, 3-dimensional, 
$n$-dimensional) subspaces of $W$ are called \emph{points} (\emph{lines}, \emph{planes}, \emph{hyperplanes}); 
in general, the $(i+1)$-dimensional subspaces are called \emph{$i$-flats}. The points of 
$\PG(n, q)$ together with the $d$-dimensional flats of $\PG(n, q)$ as blocks and  
incidence by natural containment form an incidence structure denoted by $\PG_d(n,q)$. 

\begin{proposition}[{\cite[Proposition 2.16, p. 13]{BJL99}}]
$\PG_d(n, q)$ is a block design with parameters 
$v = \qbin{n+1}{1}{q} = (q^{n+1}-1)/(q-1)$,
$k = \qbin{d}{1}{q} = (q^{d+1}-1)/(q-1)$,
$r = \qbin{n}{d}{q}$, 
$\lambda = \qbin{n-1}{d-1}{q}$, 
and
$b = \qbin{n+1}{d+1}{q}$.
\end{proposition}

\begin{remark}\label{NontrivialSymmetricDesign}  For any integer $n \ge 2$ and prime power $q$, $\PG_{n-1}(n, q)$ is a non-trivial symmetric design, that is, a design with the same number of points and blocks. Indeed, the number of points and hyperplanes in $\PG(n, q)$ coincide and equal to $\qbin{n+1}{1}{q} = \qbin{n+1}{n}{q}$.   
\end{remark}


\begin{lemma}\label{Orbitlemma}
Let $D$ be a non-trivial symmetric design, and let $G < \operatorname{Aut}D$ be an automorphism group of $D$. Then $G$ is transitive on the set of blocks of $D$ if and only 
if $G$ is transitive on the set of points of $D$. 
\end{lemma}
\begin{proof}
This is a special case of \cite[Chapter III, Theorem 4.1]{BJL99}.
\end{proof}

\subsection{Basic property of Peisert-type graphs}

\begin{lemma}\label{lem:Cayley_iso}
Let $G$ be a group and $\operatorname{Cay}(G,S_1), \operatorname{Cay}(G,S_2)$ be two Cayley graphs over this group. If there exists $\varphi \in \operatorname{Aut}G$ such that $\varphi(S_1) = S_2$, then $\operatorname{Cay}(G,S_1)$ and $\operatorname{Cay}(G,S_2)$ are isomorphic.
\end{lemma}
\begin{proof}
    If follows from the definitions of a Cayley graph and a group automorphism. 
\end{proof}

A $k$-regular graph on $v$ vertices is called \emph{strongly regular} with parameters $(v,k,\lambda,\mu)$, if any two adjacent vertices have $\lambda$ common neighbors and any two distinct non-adjacent vertices have $\mu$ common neighbors. 
If $X$ is a strongly regular graph, then its complement is also a strongly regular
graph. A strongly regular graph $X$ is \emph{primitive} if both $X$ and its complement
are connected. There are a few different approaches to show that Peisert-type graphs are strongly regular. For example, it follows from \cite[Theorem 2]{BWX99}, which is based on an explicit computation involving Gauss sums. We also refer to an elementary proof in \cite[Corollary 5]{AGLY22}.

\begin{lemma}[{\cite[Corollary 5]{AGLY22}}]\label{SRG}
A Peisert-type graph of type $(m,q)$ is strongly regular with parameters
$
(q^2, m(q - 1), (m - 1)(m - 2) + q - 2, m(m - 1))
$
and eigenvalues
$k=m(q-1)$ (with multiplicity $1$), $-m$ (with multiplicity  $q^2-1-k$) and
$q-m$ (with multiplicity $k$). 
\end{lemma}

A clique in a regular graph $X$ is called \emph{regular} if all vertices outside the clique have a constant positive number of neighbors in the clique, and this constant is called the \emph{nexus} for a regular clique. In a strongly regular graph, a clique is regular if and only if its size meets the Delsarte-Hoffman bound; moreover, the nexus of a regular clique is uniquely determined by the parameters of the strongly regular graph \cite[Proposition 1.3.2]{BCN89}. Regular cliques in strongly regular graphs are called \emph{Delsarte cliques}.
Note that in Peisert-type graphs of the same type, all canonical cliques and all non-canonical (if any) are always regular with the same nexus, and the nexus is uniquely determined by the type of a Peisert-type graph.

\subsection{Weight-distribution bound for strongly regular graphs} \label{sec:WDB}

Let $\theta$ be an eigenvalue of a graph $X$. A real-valued function $f$ on the vertex set of $X$ is called an \emph{eigenfunction}
of the graph $X$ corresponding to the eigenvalue $\theta$ (or a \emph{$\theta$-eigenfunction} of $X$), if $f \not \equiv 0$ and
for any vertex $\gamma$ in $X$ the condition
\begin{equation}\label{LocalCondition}
\theta\cdot f(\gamma)=\sum_{\substack{\delta\in{X(\gamma)}}}f(\delta)
\end{equation}
holds, where $X(\gamma)$ is the set of neighbours of the vertex $\gamma$.
We refer to the recent survey \cite{SV21} for a summary of results on the problem of finding the minimum cardinality of support of eigenfunctions of graphs and characterizing the optimal eigenfunctions.

The following lemma gives a lower bound for the number of non-zeroes (i.e., the cardinality of the support) for an eigenfunction of a strongly regular graph, known as the {\em weight-distribution bound}.  It is a special case of a more general result for distance-regular graphs \cite[Section 2.4]{KMP16}.

\begin{lemma}[{\cite[Corollary 2.3]{GSY23}}]\label{WDBsrg}
Let $X$ be a primitive strongly regular graph with non-principal eigenvalues $\theta_1$ and $\theta_2$, such that $\theta_2 < 0 < \theta_1$. Then an eigenfunction of $X$ corresponding to the eigenvalue $\theta_1$ has at least $2(\theta_1+1)$ non-zeroes, and an eigenfunction corresponding to the eigenvalue $\theta_2$ has at least $-2\theta_2$ non-zeroes.
\end{lemma}

The following lemma gives a combinatorial interpretation of the tightness of the weight-distribution bound in terms of special induced subgraphs; see also \cite[Remark 4.4]{GSY23}.

\begin{lemma}\label{WDBsrg1}
Let $X$ be a primitive strongly regular graph with eigenvalues $\theta_2 < 0 < \theta_1$. Then the following statements hold.\\
{\rm (1)} If $X$ has Delsarte cliques and each edge of $X$ lies in a constant number of Delsarte cliques (for example, $X$ is an edge-transitive strongly regular graph with Delsarte cliques), then there is an one-to-one correspondence between induced complete bipartite subgraphs with parts of size $-\theta_2$ and $\theta_2$-eigenfunctions of $X$ whose cardinality of support meets the weight-distribution bound (up to multiplication by a constant, such a function has value $1$ on the vertices of one part and value $-1$ on the vertices of the other part).\\
{\rm (2)} If the complement of $X$ has Delsarte cliques and each edge of $X$ lies in a constant number of Delsarte cliques (for example, $X$ is a coedge-transitive strongly regular graph whose complement has Delsarte cliques), then there is one-to-one correspondence between induced pairs of isolated cliques of size $-\overline{\theta_2} = \theta_1+1$ and $\theta_1$-eigenfunctions of $X$ whose cardinality of support meets the weight-distribution bound (up to multiplication by a constant, such a function has value $1$ on the vertices of one part and value $-1$ on the vertices of the other part).
\end{lemma}
\begin{proof}
(1) Since $f$ is an eigenfunction whose cardinality of support meets the weight-distribution bound and $X$ admits a Delsarte pair, it follows from 
\cite[Theorem 3(c')]{KMP16} that the support $T$ of $f$ is a complete bipartite subgraph with parts of size $-\theta_2$. It follows from \cite[Theorem 3(a')]{KMP16} that $T$ is an $S$-bitrade, where $S$ is the set of Delsarte cliques. It then follows from \cite[Theorem 1]{KMP16} that $f$ has the structure as required.

(2) Let $\overline{X}$ be the complement of $X$.
The statement follows from (1), the fact that $\theta_1$-eigenspace of $X$ coincides with $\overline{\theta_2}$-eigenspace of $\overline{X}$, where
$\overline{\theta_2} = -(1+\theta_1)$ is the negative eigenvalue of $\overline{X}$.
\end{proof}

Thus, in view of Lemma \ref{WDBsrg1}, to show the tightness of the weight-distribution bound for non-principal eigenvalues, it suffices to find a special induced subgraph (a pair of isolated cliques $T_0$ and $T_1$ or a complete bipartite graph with parts $T_0$ and $T_1$) and verify the condition from Lemma \ref{WDBsrg1}(1) or \ref{WDBsrg1}(2).

Finding the minimum cardinality of support of eigenfunctions of strongly regular graphs (in particular Peisert-type graphs) is of particular interest \cite{GKSV18, GSY23}. 
It is known that the weight-distribution bound is tight for non-principal eigenvalues of Paley graphs of square order \cite{GKSV18}. This result was generalized to generalized Paley graphs in \cite[Theorem 4.3]{GSY23}: if $d \geq 3$ and $d \mid (q+1)$, then the weight-distribution bound is tight for the negative non-principal eigenvalue in the $d$-Paley graph $\mathrm{GP}(q^2,d)=\operatorname{Cay}(\F_{q^2}^+, (\F_{q^2}^*)^d)$. Note that such $d$-Paley graphs are Peisert-type graphs \cite[Lemma 2.13]{AY22}.
We will show that if $q$ is a square, then the weight-distribution bound is tight for non-principal eigenvalues of all extremal graphs defined over $\F_{q^2}$. 

\subsection{Graphs associated with quadratic forms}
Let $V(n,r)$ be an $n$-dimensional vector space over the finite field $\mathbb{F}_r$, where $n \ge 2$ and $r$ is a prime power. Let $f(x_1,x_2,\ldots,x_n): V \to \mathbb{F}_r$ be a quadratic form on $V(n,r)$. Define a graph $Y_f$ on the set of vectors of $V(n,r)$ as follows:
$$
\text{for any~} u,v \in V(n,r), \quad  u \sim v \text{~if and only if~} f(u-v) = 0.
$$
In other words, $Y_f$ is the Cayley graph $\operatorname{Cay}(V(n,r), f^{-1}(0) \setminus \{\textbf{0}\})$.

Two quadratic forms $f_1(x_1,x_2,\ldots,x_n)$ and $f_2(y_1,y_2,\ldots,y_n)$ are said to be \emph{equivalent} if there exists an invertible matrix $B \in \GL(n,r)$ such that $f_1(Bx)=f_2(y)$.

\begin{lemma}\label{EquivalentFormsGiveIsomorphicGraphs}
    Let $f_1$ and $f_2$ be two equivalent quadratic forms. Then the graphs $Y_{f_1}$ and $Y_{f_2}$ are isomorphic.
\end{lemma}
\begin{proof}
    The linear change of variables is an isomorphism. 
\end{proof}

\subsection{Affine polar graphs}
Let $V$ be a $(2e)$-dimensional vector space over a finite field $\mathbb{F}_r$, where $e \ge 2$ and $r$ is a prime power,
equipped with the hyperbolic
quadratic form 
$$Q(x) = x_1x_2 + x_3x_4+\ldots+x_{2e-1}x_{2e}.$$
The set $Q^+$ of zeroes of $Q$ is called a non-degenerate \emph{hyperbolic quadric}. Note $e$ is known to be the maximal dimension of a subspace in $Q^+$. A \emph{generator} of $Q^+$ is a maximal totally isotropic subspace in $Q^+$.

Denote by $VO^+(2e,r)$ the graph on $V$ with two vectors $x,y$ being adjacent if and only if $Q(x-y) = 0$. The graph $VO^+(2e,r)$ is known as an \emph{affine polar graph} (see \cite{B1, BV22}).

\begin{lemma}\label{MaximalCliquesInVO}
There is a one-to-one correspondence between generators of $Q^+$ and maximal cliques in $VO^+(2e,r)$ containing the vector $0$. 
\end{lemma}

\begin{proof}
By \cite[Lemma 2.10]{EGP19}, there is a one-to-one correspondence between cosets of generators of $Q^+$ and maximal cliques in $VO^+(2e,r)$. Note that the graph $VO^+(2e,r)$ is vertex-transitive, so the lemma follows. 
\end{proof}

\begin{lemma}\label{lem:maximalclique}
The graph $VO^+(4,r)$ has exactly $2(r+1)$ maximal cliques containing zero vector; these are the generators.    
\end{lemma}
\begin{proof}
In view of Lemma \ref{MaximalCliquesInVO}, it suffices to show the number of generators containing $0$ is $2(r+1)$. Equivalently, we need to count the number of generators of the polar space associated with a nondegenerate hyperbolic quadratic form over $\PG(3,r)$. Such a space has rank $2$ with parameter $0$ (see for example \cite[Table 1]{DD18}). Thus, the lemma follows from the first statement of \cite[Lemma 1.3]{DD18}.
\end{proof}

\begin{lemma}[]\label{lem:equivalent}
All maximal cliques of an affine polar graph $VO^+(4,r)$ are equivalent under the action of the automorphism group.    
\end{lemma}
\begin{proof}
Let $C$ be a maximal clique. Since translations are automorphisms, we may assume that $0 \in C$ without loss of generality. By Lemma \ref{lem:maximalclique}, $C$ is a generator.
The subgroup fixing 0 is the projective group and it is
transitive on generators (the generators form a single orbit under the action of the automorphism group of the hyperbolic polar space \cite[Sections 2.3.5, 2.6.5 and 2.6.6]{BV22}), which implies the result.
\end{proof}

\section{The parameters of extremal Peisert-type graphs} \label{parameter}

In this section, we determine the parameters of extremal Peisert-type graphs with the help of the theory of directions.

\subsection{Direction set determined by a Peisert-type graph}
Let $U$ be a point set in $\AG(2,q)$; the \emph{set of directions determined} by $U$ is defined to be
$$
D (U):=  \left\{ [a-c: b-d]:  (a,b), (c,d) \in U,\, (a,b) \neq (c,d) \right\} \subset \PG(1,q).
$$
Let $v \in \F_{q^2} \setminus \F_q$ so that $\{1,v\}$ forms a basis 
of $\F_{q^2}$ over $\F_q$; then we can identify $\F_{q^2}$ and $\AG(2,q)$ via the embedding $\pi \colon \F_{q^2}\to \AG(2,q)$, where 
$$
\pi(a+bv)=(a,b), \quad \forall a,b \in \F_q.
$$
Also define the map $\sigma \colon \AG(2,q)\setminus \{(0,0)\} \to \PG(1,q)$ such that
$$
\sigma((a,b))=[a:b], \quad \forall (a,b) \in \AG(2,q)\setminus \{(0,0)\}.
$$

Let $X$ be a Peisert-type graph $X=\operatorname{Cay}(\F_{q^2}^+,S)$ of type $(m,q)$. Following \cite{Y24}, we define \emph{the direction set determined by $X$} to be $\mathcal{D}(X):=\sigma(\pi(S))$. The following lemma justifies this terminology, in the sense that it converts the ``clique problem" into the ``direction problem".

\begin{lemma}[{\cite[Proposition 4.1]{Y24}}]
\label{lem:correspondence}
Let $X=\operatorname{Cay}(\F_{q^2}^+,S)$ be a Peisert-type graph  and let $C \subset \F_{q^2}$. Then $C$ is a clique in $X$ if and only if $D(\pi(C)) \subset \mathcal{D}(X)$, and $C$ is a canonical clique in $X$ only if $|D(\pi(C))|=1$.
\end{lemma}

\begin{lemma}\label{lem:induce_isomorphism}
Let $X, Y$ be Peisert-type graphs of type $(m,q)$. If $\mathcal{D}(X)$ and $\mathcal{D}(Y)$ are projectively equivalent, then $X$ and $Y$ are isomorphic. 
\end{lemma}
\begin{proof}
Let $S$ and $T$ be the connection set of $X$ and $Y$, respectively. Since $\mathcal{D}(X)$ and $\mathcal{D}(Y)$ are projectively equivalent, it follows that $\pi(S)$ and $\pi(T)$ are affinely equivalent. In particular, there is an automorphism $\varphi$ of $\F_{q^2}$ such that $\varphi(S)=T$. Thus, it follows from Lemma~\ref{lem:Cayley_iso} that $X$ and $Y$ are isomorphic.   
\end{proof}

\begin{corollary}\label{cor:3iso}
Given a prime power $q$, all Peisert-type graphs of type $(3,q)$ are isomorphic.
\end{corollary}

\begin{proof}
Recall that $\PGL(2,q)$ acts naturally on $\PG(1,q)$ by right multiplication as follows:
$$A = \begin{pmatrix} a &b \\ c &d \end{pmatrix} \mbox{ sends } [u_1: u_2] \mbox{ to } [au_1+bu_2: cu_1+du_2].$$ 
It is well known that this action is $3$-transitive (see for example \cite[p. 245]{DM96}). Therefore, by Lemma~\ref{lem:induce_isomorphism}, all Peisert-type graphs of type $(3,q)$ are isomorphic.
\end{proof}

The following lemma is classical in the theory of directions; see also \cite[Theorem 2.4]{Y24}.
\begin{lemma}\label{lem:p}
Let $U$ be a subset of $\AG(2,p)$ with $|U|=p$. 

(1) (R\'{e}dei and Megyesi \cite{R73}) If the points in $U$ are not all collinear, then $U$ determines at least $\frac{p+3}{2}$ directions.

(2) (Lov\'{a}sz and Schrijver \cite{LS83}) $U$ determines exactly $(p+3) / 2$ directions if and only if $U$ is affinely equivalent to the set $\{(x,x^{(p+1)/2}): x \in \F_p\}$ . 
\end{lemma}

Let $K$ be a subfield of $\F_q$. We say $U \subset \AG(2,q)$ is {\em $K$-linear} if there exists $\alpha, \beta \in \F_{q^2}$ that are linearly independent over $\F_q$, such that $W=\{\alpha x+\beta y: (x,y) \in U\}$ forms a $K$-subspace in $\F_{q^2}$. The following result, due to Blokhuis, Ball, Brouwer, Storme, and Sz{\H{o}}nyi \cite{B03, BBBSS}, predicts the algebraic structure of $U$ based on the size of $|D(U)|$.

\begin{lemma}[\cite{B03, BBBSS}]\label{lem:Ball}
Let $U \subset \AG(2,q)$ be a point set with $q$ points such that $(0,0) \in U$ and $N=|D(U)|< \frac{q}{2}+1$. Then there is a subfield $K$ of $\F_q$ such that $U$ is $K$-linear. Moreover, if $K$ is the largest subfield over which $U$ is $K$-linear and $K \neq \F_q$, then $N \geq q/|K|+1$.
\end{lemma}
\begin{proof}
This is a simplified version of the main theorem in \cite{B03}. Note that the main theorem in \cite{B03} assumes that $U$ is the graph of the function $f\colon \F_q \to \F_q$ such that $f(0)=0$. However, it is well known that any set of $q$ points in $\AG(2,q)$ that does not determine all the directions are affinely equivalent to the graph of a function (see for example \cite[page 342]{B03}).
\end{proof}

Next, we combine Lemma~\ref{lem:correspondence} and Lemma~\ref{lem:Ball} to deduce a useful proposition, which is crucial in the proof of our main results. This proposition is implicit in \cite{AY22} and could be used to deduce the strict-EKR property of a family of Peisert-type graphs. For the sake of completeness, we include a short proof.

\begin{proposition}\label{prop:subspace}
Let $q=p^n$ with $n>1$. Let $k$ be the largest proper divisor of $n$. We have the following:
\begin{itemize}
    \item Any Peisert-type graph of type $(m,q)$ has the strict-EKR property provided that 
$m \leq p^{n-k}$.
    \item If $p=2$, additionally assume that $n$ is not a prime; if $p\geq 3$, this assumption is not needed. Then in a Peisert-type graph of type $(p^{n-k}+1,q)$, each maximum clique containing $0$ is an $\F_{p^k}$-subspace in $\F_{q^2}$.  
\end{itemize}
\end{proposition}
\begin{proof}
Let $X$ be a Peisert-type graph of type $(m,q)$, where $m \leq p^{n-k}+1$. We claim that $m<\frac{q}{2}+1$. Indeed, if $p \geq 3$, then clearly $m\leq \frac{q}{p}+1<\frac{q}{2}+1$; if $p=2$, then $k \geq 2$ and thus $m\leq \frac{q}{4}+1<\frac{q}{2}+1$.

Suppose $X$ fails to have the strict-EKR property. Let $C$ be a non-canonical clique with $0 \in C$. Since $C$ is a non-canonical clique, Lemma~\ref{lem:correspondence} implies that $D(\pi(C)) \subset \mathcal{D}(X)$ and $|D(\pi(C))|>1$. In view of Lemma~\ref{lem:Ball}, $\pi(C)$ is $K$-linear for some subfield $K$ of $\F_q$ and we take $K$ to be the largest such subfield. Since $C$ is non-canonical, we have $K \neq \F_q$. Thus $|K| \leq p^k$ by the assumption that $k$ is the largest proper divisor of $n$. Lemma~\ref{lem:Ball} then implies that 
\begin{equation}\label{eq:Ball}
m=|\mathcal{D}(X)| \geq  |D(\pi(C))| \geq q/|K|+1 \geq p^{n-k}+1.  
\end{equation}
Thus, $m=p^{n-k}+1$ and we must have $K=\F_{p^k}$, equivalently, $C$ is an $\F_{p^k}$-subspace. This proves the two statements of the proposition.
\end{proof}

\subsection{Proof of Theorem~\ref{thm:thm1}}\label{sec:thm1}

In this section, we determine the parameters of extremal Peisert-type graphs and prove Theorem~\ref{thm:thm1}.

\begin{proof}[Proof of Theorem~\ref{thm:thm1}]
First, we consider the case $n=1$, that is, $q=p$ is a prime. Let $X$ be of type $(m,p)$.

We begin by showing that $m \leq \frac{p+3}{2}$ by constructing a Peisert-type graph of type $(\frac{p+3}{2},p)$ without the strict-EKR property.  Let $H$ be the set of non-zero squares in $\F_p$, and let $u \in \F_{p^2} \setminus \F_p$. Consider the Peisert-type graph $Y$    defined over $\F_{p^2}$ with connection set $$S=\F_p^* \cup u\F_p^* \cup \cup_{h \in H} (u-h)\F_p^*.$$ Since $|H|=\frac{p-1}{2}$, $Y$ is of type $(\frac{p+3}{2},p)$. Based on the connection set of $Y$, it is easy to verify that $H \cup uH \cup \{0\}$ is a non-canonical clique in $Y$, so $Y$ does not have the strict-EKR property.
    
Now, let $C$ be a non-canonical clique in $X$ and let $U=\pi(C)$. By Lemma~\ref{lem:correspondence}, $|D(U)|>1$. In particular, the $p$ points in $U$ are not collinear and thus we have $m=|\mathcal{D}(X)| \geq |D(U)|\geq \frac{p+3}{2}$ by Lemma~\ref{lem:correspondence} and by Lemma~\ref{lem:p} (1). Therefore, $m=\frac{p+3}{2}$. Moreover, we must also have $\mathcal{D}(X)=D(U)$ by Lemma~\ref{lem:correspondence}, and thus $U$ is affinely equivalent to the set $U_0=\{(x,x^{(p+1)/2}): x \in \F_p\}$ by Lemma~\ref{lem:p} (2).  Let $X'$ be another Peisert-type graph of type $(\frac{p+3}{2},p)$ without the strict-EKR property. Let $C'$ be a non-canonical clique in $X'$ and let $U'=\pi(C')$. Then by the same argument, $\mathcal{D}(X')=D(U')$ and $U'$ is affinely equivalent to $U_0$. It follows that $D(U)$ and $D(U')$ are projectively equivalent, that is,  $\mathcal{D}(X)$ and $\mathcal{D}(X')$ are projectively equivalent. Now Lemma~\ref{lem:induce_isomorphism} implies that $X$ and $X'$ are isomorphic, establishing the uniqueness of the extremal graph. 

Next, we discuss the case where $n>1$. Let $d=n/k$. In view of Proposition~\ref{prop:subspace}, it suffices to construct a Peisert-type graph of type $(p^{n-k}+1,q)$ without the strict-EKR property.  Let $U \subset \F_q$ be a $(d-1)$-dimensional subspace over $\F_{p^k}$. And let $x \in \F_{q^2} \setminus \F_q$. Consider the Peisert-type graph $X=\operatorname{Cay}(\F_{q^2}^+,S)$, where
$$
S=\F_q^* \sqcup \bigsqcup_{u \in U} (u+x)\F_q^*.
$$
Next, we justify that the above union is disjoint, which implies that $X$ is of type $(p^{n-k}+1,q)$. It is clear that $\F_q^* \neq (u+x)\F_q^*$ for any $u \in U$. Suppose that $u,u' \in U$ such that $(u+x)\F_q^*=(u'+x)\F_q^*$. Then there is $y \in \F_q^*$, such that $u+x=y(u'+x)$, which implies that $u-yu'=(y-1)x$. Note that $u-yu' \in \F_q$ and $x \notin \F_q$. Thus, we must have $y=1$ and thus $u=u'$. 

It remains to show that $X$ does not have the strict-EKR property. Let $V=U \oplus x\F_{p^k}$. Note that $S=(V\setminus \{0\})\F_q^*$. In particular, we have $V-V=V \subset S \cup \{0\}$ and thus $V$ is a non-canonical clique in $X$.
\end{proof}

\begin{remark}\label{rem:direction}
In view of Lemma~\ref{lem:correspondence} and the above proof of Theorem~\ref{thm:thm1}, our main results (to be proved in later sections) imply an analogue of Lemma~\ref{lem:p} (2) when $q$ is a square or a cube (instead of a prime). In other words, our main results strengthen Lemma \ref{lem:Ball} due to Blokhuis, Ball, Brouwer, Storme, and Sz{\H{o}}nyi \cite{B03, BBBSS} in the sense that we classify non-collinear sets $U \subset \AG(2,q)$ with $q$ points that determines the minimum number of directions when $q$ is a square or a cube, as well as such direction sets. While the theory of directions provides useful tools towards the study of the strict-EKR property of Peisert-type graphs, typically additional ingredients and insights are needed for the latter.
\end{remark}

\subsection{A non-canonical clique uniquely determines an extremal graph}

\begin{corollary}\label{cor:S}
Let $q=p^n \geq 3$ be a prime power such that $n>1$. If $n$ is a prime, further assume that $p>2$. Let $X$ be an extremal Peisert-type graph defined over $\F_{q^2}$. Then a non-canonical clique $V$ containing $0$ determines the connection set $S$ of $X$ by $S=(V\setminus \{0\})\F_q^*$.    
\end{corollary}
\begin{proof}
Let $k$ be the largest proper divisor of $n$. Then Theorem~\ref{thm:thm1} implies that $X$ is of type $(p^{n-k+1},q)$. Let $V$ be a non-canonical clique in $X$ such that $0 \in V$ and let $S$ be the connection set of $X$. It follows that $V \setminus \{0\} \subset S$. Moreover, since $S$ is a union of $\F_q^*$-cosets, we have $S':=(V\setminus \{0\})\F_q^* \subset S$. Suppose that $S' \neq S$. Consider the Peisert-type graph $X'=\operatorname{Cay}(\F_{q^2}^+,S')$; then $X'$ is of type $(m,q)$ for some $m \leq p^{n-k}$. However, note that Proposition~\ref{prop:subspace} implies that $X'$ has the strict-EKR property, contradicting the fact that $V$ is a non-canonical clique in $X'$.
\end{proof}

When $q$ is a square, the following proposition follows from Proposition~\ref{prop:subspace} and Corollary~\ref{cor:S}.

\begin{proposition}\label{prop:2}
Let $q$ be a square such that $q \geq 9$.  Let $X$ be an extremal Peisert-type graph defined over $\F_{q^2}$. Then the connection set $S$ of $X$ is of the form $S=(a \F_{\sqrt{q}} \oplus b \F_{\sqrt{q}}\setminus \{0\})\F_q^*$, where $a, b \in \F_{q^2}$ are linearly independent over $\F_q$. 
\end{proposition}

\begin{proof}
By Theorem~\ref{thm:thm1}, $X$ is of type $(\sqrt{q}+1, q)$. Let $V$ be a non-canonical clique containing $0$. By Proposition~\ref{prop:subspace}, $V$ is a $2$-dimensional subspace over $\F_{\sqrt{q}}$. Since $V$ is non-canonical, we can find $a, b \in \F_{q^2}$ are linearly independent over $\F_q$ such that $V=a \F_{\sqrt{q}} \oplus b \F_{\sqrt{q}}$. The proposition then follows from Corollary~\ref{cor:S}.
\end{proof}

When $q$ is a cube, we prove a result of a similar flavor below.

\begin{proposition}\label{prop:3}
Let $q=r^3$, where $r \geq 3$ is a non-square. Let $X$ be an extremal Peisert-type graph defined over $\F_{q^2}$. Then the connection set $S$ of $X$ is of the form $S=(a \F_{r} \oplus b \F_{r} \oplus c \F_{r} \setminus \{0\})\F_q^*$, where $a, b,c \in \F_{q^2}$ such that $a,b$ are linearly dependent over $\F_q$, and $a,c$ are linearly independent over $\F_q$. 
\end{proposition}

\begin{proof}
Let $q=p^n$. Let $k$ be the largest proper divisor of $n$. Since $q$ is not a square and $q$ is a cube, it follows that $k=n/3$ and $r=p^k$. By Theorem~\ref{thm:thm1}, $X$ is of type $(r^2+1, q)$ since $p^{n-k}+1=r^2+1$. Let $V$ be a non-canonical clique containing $0$. By Proposition~\ref{prop:subspace}, $V$ is a $3$-dimensional subspace over $\F_{p^k}=\F_r$. By Corollary~\ref{cor:S}, we have $S=(V\setminus \{0\})\F_q^*$. Thus, it suffices to analyze the structure of $V$.

We can write $S=\sqcup_{i=1}^{r^2+1} c_i \F_q^*$. For each $1 \leq i \leq r^2+1$, it is clear that $|c_i \F_q \cap V|\in \{r,r^2\}$ since $c_i \F_q \cap V$ is a nonempty vector space over $\F_{r}$ and $V$ is a non-canonical clique. Let $A$ be the number of $i$ such that $|c_i \F_q \cap V|=r$. Then we have
$$
r^3-1=|V|-1=|S \cap V|=\sum_{i=1}^{r^2+1} |c_i \F_q^* \cap V|=\sum_{i=1}^{r^2+1} (|c_i \F_q \cap V|-1)=A(r-1)+(r^2+1-A)(r^2-1),
$$
and thus $A=r^2$. Therefore, there is exactly one $i_0$ such that $|c_{i_0}\F_q \cap V|=r^2$. Let $U=c_{i_0} \F_q \cap V$. Then we can write $U=a\F_r \oplus b\F_r$, where $a/b \in \F_q$ and $a/b \notin \F_r$. Pick $c \in V \setminus U$; then we have $V=a\F_r \oplus b\F_r \oplus c\F_r$, and $c/a \notin \F_q$. Since $S=(V\setminus \{0\})\F_q^*$, the result follows. 
\end{proof}

\section{Classification of non-canonical cliques in extremal graphs}\label{sec:count}

\subsection{$q$ is a square} In this section, we classify maximum cliques in an extremal graph defined over $\F_{q^2}$, where $q$ is a square. 

\begin{proposition}\label{prop:NonCanonicalCliques}
Let $q$ be a square such that $q \geq 9$. Let $X$ be an extremal Peisert-type graph defined over $\F_{q^2}$. Then $X$ has exactly $\sqrt{q}+1$ non-canonical cliques containing $0$.  
\end{proposition}
\begin{proof}
Let $S$ be the connection set of $X$.
By Proposition~\ref{prop:2}, we can find $a, b \in \F_{q^2}$ that are linearly independent over $\F_q$, such that the connection $S$ can be written as:
\begin{equation}\label{decompsss}
S=(V \setminus \{0\}) \F_q^*=b \F_q^* \cup \bigcup_{\lambda \in \F_{\sqrt{q}}} (a+\lambda b)\F_q^*,    
\end{equation}
where $V=a\F_{\sqrt{q}} \oplus b\F_{\sqrt{q}}$. Note that $V$ is a non-canonical clique in the graph $X$. 

Let $\varepsilon$ be a primitive element of $\mathbb{F}_q$; then $\varepsilon^{\sqrt{q}+1}$ is a primitive element of the subfield $\mathbb{F}_{\sqrt{q}}$. From the definition of $S$, it is clear that $\varepsilon^{i} V$ is a non-canonical clique in the graph $X$, for each $i \in \{0,1,\ldots,\sqrt{q}\}$. We claim that for each $0 \leq i,j \leq \sqrt{q}$ such that $i \neq j$, we have $\varepsilon^{i} V \cap \varepsilon^{j} V=\{0\}$. Suppose otherwise, then there exist $x,y \in V \setminus \{0\}$ such that $\varepsilon^i x=\varepsilon^jy$, that is, $\varepsilon^{i-j} x=y$. It follows that $x,y$ are linearly dependent over $\F_{q}$ and $x,y$ are linearly independent over $\F_{\sqrt{q}}$. Thus, $\{x,y\}$ forms a basis of $V$ over $\F_{\sqrt{q}}$, which implies that $V$ is a canonical clique, a contradiction.  In summary, 
we have shown that $$V \setminus \{0\}, \varepsilon V \setminus \{0\}, \ldots, \varepsilon^{\sqrt{q}} V \setminus \{0\}$$ forms a partition of the connection set $S$, and $V, \varepsilon V, \ldots \varepsilon^{\sqrt{q}} V$ are $\sqrt{q}+1$ non-canonical cliques in $X$ containing 0. Next, we show there is no other non-canonical clique in $X$.

Suppose that $V'$ is a non-canonical clique in $X$ containing 0. From Proposition~\ref{prop:subspace} and Proposition~\ref{prop:2}, $V'$ is a $2$-dimensional subspace over $\F_{\sqrt{q}}$, and $S=(V' \setminus \{0\}) \F_q^*$. Since $a \F_q^*, b \F_q^* \subset S$, there exist $\mu, \nu \in \F_q^*$ such that $\mu a, \nu b \in V'$. Recall that $a/b \notin \F_q$, thus $V'=\mu a\F_{\sqrt{q}} \oplus \nu b\F_{\sqrt{q}}$. Since $(\mu a+\nu b)\F_q^* \subset S$, in view of equation~\eqref{eq:decompSS}, we have $\mu a+\nu b =x(a+\lambda b)$ for some $\lambda \in \F_{\sqrt{q}}^*$ and $x \in \F_q^*$. It follows that $(\mu-x)a=(\lambda x -\nu)b$. Note that both $\mu-x$ and $\lambda x-\nu$ are in $\F_q$, while $a/b \notin \F_q$. Thus, we have $\mu=x$ and $\lambda x=\nu$. It follows that $\nu/\mu=\lambda \in \F_{\sqrt{q}}$ and thus $V'=\mu V$. Recall that $\varepsilon$ is a primitive root of $\F_q$, $\mu \in \F_q$, and $\varepsilon^{\sqrt{q}+1}V=V$. Thus, there is $0 \leq i \leq \sqrt{q}$ such that $V'=\varepsilon^i V$. This completes the proof of the theorem. 
\end{proof}


Next, we deduce two interesting corollaries.

\begin{corollary}\label{cor:numberofgraphs}
Let $q$ be a square such that $q \geq 9$. Then the number of extremal Peisert-type graphs defined over $\F_{q^2}$ is given by $(q+1)\sqrt{q}$.
\end{corollary}

\begin{proof}
We count the total number of non-canonical cliques containing 0 among all extremal graphs defined over $\F_{q^2}$. According to Proposition~\ref{prop:3}, each such clique is a $2$-dimensional subspace over $\F_{\sqrt{q}}$ which is not a $1$-dimensional subspace over $\F_q$. Thus, the total number of such non-canonical cliques is equal to the difference of the following two Gaussian coefficients:
\begin{align*}
 \qbin{4}{2}{\sqrt{q}}- \qbin{2}{1}{q}
&= \frac{(q^2-1)(q^2-\sqrt{q})}{(q-1)(q-\sqrt{q})}-\frac{q^2-1}{q-1}\\
&=(q+1)(q+\sqrt{q}+1)-(q+1)=\sqrt{q}(\sqrt{q}+1)(q+1).
\end{align*}

Note that Proposition~\ref{prop:3} also states that a non-canonical clique in an extremal graph uniquely determines the connection set of the graph, and thus determines the graph. By Proposition~\ref{prop:NonCanonicalCliquesCube}, each extremal graph has $\sqrt{q}+1$ non-canonical cliques. Thus, the corollary follows immediately.
\end{proof}

\begin{corollary}
\label{cor:3induce}
   Let $q$ be a square such that $q \geq 9$. Then for any given Peisert-type graph $Z$ of type $(3,q)$, there exists a unique extremal Peisert-type graph defined over $\F_{q^2}$ that contains $Z$ as a subgraph. 
\end{corollary}
\begin{proof}
Let $X$ be an extremal graph defined over $\F_{q^2}$. Then by Theorem~\ref{thm:thm1}, $X$ is of type $(\sqrt{q}+1,q)$ and thus it contains $\binom{\sqrt{q}+1}{3}$ many Peisert-type graphs of type $(3,q)$ as subgraphs. On the other hand, if $Z,Z'$ are two Peisert-type graphs of type $(3,q)$, then the number of extremal graphs containing $Z$ as a subgraph and the number of extremal graphs containing $Z'$ as a subgraph are equal in view of the proof of Corollary~\ref{cor:3iso}. 

Note that
$$
\frac{\binom{q+1}{3}}{\binom{\sqrt{q}+1}{3}}=\frac{(q+1)q(q-1)}{(\sqrt{q}+1)\sqrt{q}(\sqrt{q}-1)}=(q+1)\sqrt{q}.
$$    
Thus, the corollary follows from Corollary~\ref{cor:numberofgraphs}.
\end{proof}

\subsection{$q$ is a cube} In this section, we classify maximum cliques in an extremal graph defined over $\F_{q^2}$, where $q$ is a cube. 

\begin{proposition}\label{prop:NonCanonicalCliquesCube}
Let $q=r^3$, where $r \geq 3$ is a non-square. Let $X$ be an extremal Peisert-type graph defined over $\F_{q^2}$. Then $X$ has exactly $r^2+r+1$ non-canonical cliques containing $0$.  
\end{proposition}

\begin{proof}
By Proposition~\ref{prop:3}, the connection set $S$ of $X$ is of the form $S=(a \F_{r} \oplus b \F_{r} \oplus c \F_{r} \setminus \{0\})\F_q^*$, where and $a, b,c \in \F_{q^2}$ such that $a,b$ are linearly dependent over $\F_q$, and $a,c$ are linearly independent over $\F_q$. Without loss of generality, we may assume that $a=1$. Then we have $b \in \F_{r^3} \setminus \F_r$ and $c \in \F_{q^2} \setminus \F_q$. In particular, since $\F_{r^2} \cap \F_{r^3}=\F_r$, it follows that $1,b,b^2$ are linearly independent over $\F_r$. 

Let $U=\F_r \oplus b\F_{r}$. We claim that if $x \in \F_q$, then $Ux=U$ holds if and only if $x \in \F_r^*$. Indeed, if $Ux=U$, then $x \in U$. If $x \notin \F_r$, then we can write $x=d+eb$ for some $d, e \in \F_r$ with $e \neq 0$. However, note that $db+eb^2=bx \in Ux$ implies that $1,b,b^2$ are linearly dependent over $\F_r$, a contradiction. 

Obviously, $V=\F_{r} \oplus b \F_{r} \oplus c \F_{r}$ is a non-canonical  clique in $X$. Note that we can also write $V$ as the disjoint union
\begin{equation}\label{eq:decomp}
V=U \sqcup \bigsqcup_{u \in U} (u+c) \F_r^*.   
\end{equation}
In particular, since $c \notin \F_q$, we have $U=V \cap \F_q$. Also note that equation~\eqref{eq:decomp} allows us to write
\begin{equation}\label{eq:decompS}
S=(V \setminus \{0\})\F_q^*=\F_q^* \sqcup \bigsqcup_{u \in U} (u+c) \F_q^*
\end{equation}
in view of Corollary~\ref{cor:S}.

Let $\varepsilon$ be a primitive element of $\mathbb{F}_q$; then $\varepsilon^{r^2+r+1}$ is a primitive element of the subfield $\mathbb{F}_{r}$. From the definition of $S$, it is clear that $\varepsilon^{i} V$ is a non-canonical  clique in the graph $X$, for each $i \in \{0,1,\ldots,r^2+r\}$. We claim that for each $0 \leq i,j \leq r^2+r$ such that $i \neq j$, we have $\varepsilon^{i} V \neq  \varepsilon^{j} V$. Suppose otherwise, then there is $1 \leq i \leq r^2+r$, such that $\varepsilon^i V=V$. We have a decomposition for $\varepsilon^i V$ which is similar to equation~\eqref{eq:decomp}:
$$
\varepsilon^i V=\varepsilon^i U \sqcup \bigsqcup_{u \in U} \varepsilon^i (u+c) \F_q^*   
$$
so that we can read that $\varepsilon^i V \cap \F_q=\varepsilon^i U$. Thus, we must have $$U=V 
\cap \F_q=\varepsilon^i V \cap \F_q=\varepsilon^i U,$$ which implies that $\varepsilon^i \in \F_r$ by the above claim, a contradiction.

We have shown that there are at least $(r^2+r+1)$ non-canonical cliques in $X$. It suffices to show that there are at most $(r^2+r+1)$ such cliques. Let $V'$ be a non-canonical clique in $X$. We can write $V'=\alpha\F_r \oplus \beta \F_r \oplus \gamma \F_r$, where $\alpha, \beta$ are linearly dependent over $\F_q$, and $\alpha, \gamma$ are linearly independent over $\F_q$. It follows from Corollary~\ref{cor:S} that 
$S=(\alpha\F_r \oplus \beta \F_r \oplus \gamma \F_r\setminus \{0\})\F_q^*$. Let $U'=\alpha\F_r \oplus \beta \F_r$. We will complete the proof by dividing our discussion according to the following two cases, namely, $\alpha \in \F_q$ or $\alpha \notin \F_q$. In Case 1, we will show the number of choices of $V$ is at most $r^2+r+1$, and we will show Case 2 is impossible.

\textbf{Case 1}: $\alpha \in \F_q$. It follows from the assumption that $\beta \in \F_q$ and $\gamma \not\in \F_q$. Similar to equation \eqref{eq:decompS}, we have
\begin{equation}\label{eq:decompSS}
S=\F_q^* \sqcup \bigsqcup_{u \in U'} (u+\gamma) \F_q^*.
\end{equation}

Note that the number of possible choices of $U'$ is the number of $2$-dimensional $\F_r$-subspace in $\F_q=\F_{r^3}$, which is given by $\frac{r^3-1}{r-1}=r^2+r+1$. Next we show that each $U'$ corresponds to at most one $V'$, which implies the total number of $V'$ in this case is at most $(r^2+r+1)$. Suppose otherwise, than we have $W_1=U' \oplus \gamma_1 \F_r$ and $W_2=U' \oplus \gamma_2 \F_r$ being two different cliques in $X$, such that $\gamma_1, \gamma_2 \notin \F_q$. In view of equation~\eqref{eq:decompSS}, we have $(U'+\gamma_1)\F_q^*=(U'+\gamma_2)\F_q^*$. Therefore, for each $u \in U'$, there is $x_u \in \F_q^*$ and $v_u \in U'$, such that
$$
u+\gamma_1=x_u(v_u+\gamma_2).
$$
Pick any $u_1, u_2 \in U'$ such that $u_1 \neq u_2$. We have
$$
u_1-u_2=(u_1+\gamma_1)-(u_2+\gamma_1)=(x_{u_1} v_{u_1}-x_{u_2} v_{u_2})+(x_{u_1}-x_{u_2})\gamma_2.
$$
In particular, this implies that
$$
(x_{u_1}-x_{u_2})\gamma_2=u_1-u_2-(x_{u_1} v_{u_1}-x_{u_2} v_{u_2}) \in \F_q^*,
$$
which forces $x_{u_1}=x_{u_2}$ since $\gamma_2 \notin \F_q$. This implies that there is $x \in \F_q^*$ such that $x=x_u$ for all $u \in U'$. 

Let $u \in U'$. We have
$$
u=(u+\gamma_1)-\gamma_1=x(v_u+\gamma_2)-x(v_0+\gamma_2)=x(v_u-v_0) \in xU'.
$$
Thus, we must have $U'=xU'$. From the claim, $x \in \F_r^*$. This implies that $(U'+\gamma_1)\F_r^*=(U'+\gamma_2)\F_r^*$ and thus $W_1=W_2$, a contradiction.

\textbf{Case 2}: $\alpha \notin \F_q$. It follows from the assumption that $\beta \not \in \F_q$. Then we can always assume that $\gamma \in \F_q$ (otherwise, there is $x,y \in \F_r$ such that $x\alpha+y\beta+\gamma \in \F_q$ in view of the decomposition~\eqref{eq:decompS}, and we can replace $\gamma$ with $\gamma-(x\alpha+y\beta)$). We will derive a contradiction.  

Since $c\F_q^* \subset S$, in view of the decomposition~\eqref{eq:decompS}, we can write $c=(v+\gamma)x_0$ for some $v \in U'$ and $x_0 \in \F_q^*$. Rewriting equation~\eqref{eq:decompS} in the following:
$$
S=\F_q^* \sqcup \bigsqcup_{u \in U} (u+c) \F_q^*=\F_q^* \sqcup \bigsqcup_{u \in U} (u+(v+\gamma)x_0) \F_q^*=\F_q^* \sqcup \bigsqcup_{u \in U} (u/x_0+v+\gamma) \F_q^*.
$$
Note that $v \F_q^* \subset S$ and $v \notin \F_q^*$, which implies that $u/x_0+\gamma=0$ for some $u \in U$. In other words, we have $U/x_0+\gamma=U/x_0$. Therefore,
$$
S=\F_q^* \sqcup \bigsqcup_{u \in U} (u/x_0+v) \F_q^*
=\F_q^* \sqcup \bigsqcup_{u \in U} (u+vx_0) \F_q^*
$$
It follows that
$$
\F_q^* \sqcup \bigsqcup_{u \in U}(u+vx_0)\F_q^*=S=(V' \setminus \{0\})\F_q^*=v \F_q^* \sqcup \bigsqcup_{u' \in U'} (u'+\gamma)\F_q^*
$$
and thus
$$
\bigsqcup_{u \in U\setminus \{0\}}(u+vx_0)\F_q^*=\bigsqcup_{u \in U'\setminus \{0\}}(u'+\gamma)\F_q^*.
$$
This allows us to find a bijection between $U \setminus \{0\}$ and $U' \setminus \{0\}$ by $u \mapsto v_u$, namely, for each $u \in U \setminus \{0\}$, there is $v_u \in U' \setminus \{0\}$ and $y_u \in U'$, such that
\begin{equation}\label{bijection}
u+vx_0=y_u(v_u+\gamma)
\end{equation}
so that $(u+vx_0)\F_q^*=(v_u+\gamma)\F_q^*$.

Let $u_1,u_2 \in U \setminus \{0\}$. From equation~\eqref{bijection}, 
$$
u_1-u_2=(u_1+vx_0)-(u_2+vx_0)=(y_{u_1}v_{u_1}-y_{u_2}v_{u_2})+\gamma(y_{u_1}-y_{u_2}),
$$
which implies that 
$$
u_1-u_2-\gamma(y_{u_1}-y_{u_2})=y_{u_1}v_{u_1}-y_{u_2}v_{u_2}.
$$
Note that the left-hand side of the above equation is in $\F_q$, and the right-hand side is in $\alpha F_q$. Thus, both sides must be equal to $0$, and thus $y_{u_1}v_{u_1}=y_{u_2}v_{u_2}.$ Consequently, for each $u \in U \setminus \{0\}$, we have
$$
y_u v_u=y_1 v_1
$$
and
$$
u-1-\gamma(y_u-y_1)=0,
$$
which implies that
\begin{equation}\label{v_u}
v_u=\frac{y_1v_1}{y_u}=\frac{y_1v_1}{\frac{u-1}{\gamma}+y_1}=\frac{y_1v_1\gamma}{u-1+y_1\gamma}.    
\end{equation}

Recall that the map $u \mapsto v_u$ defines a bijection from $U \setminus \{0\}$ to $U' \setminus \{0\}$. Therefore, equation~\eqref{v_u} implies that
$$
U'\setminus \{0\}=\frac{y_1v_1\gamma}{(U \setminus \{0\})-1+y_1\gamma}.
$$
In particular, $(U \setminus \{0\})-1+y_1\gamma$ is closed under multiplication over $\F_r^*$ since $U'$ is a subspace over $\F_r$. It follows that $y_1\gamma-1 \in U$ since $r \geq 3$. Thus,
$$
U''\setminus \{0\}=\frac{1}{U \setminus \{y_1\gamma-1\}},
$$
where $U''=U'/(y_1v_1\gamma)$.
From this, it is easy to derive a contradiction. For simplicity, we only consider the case that $y_1\gamma-1 \not \in \{1,b, b+1\}$; the arguments for the other cases are similar. Under this setting, we have $1, 1/b, 1/(b+1) \in U''$. Note that $1, 1/b$ forms a basis of $U''$ over $\F_r$. Thus, there exist $r,s \in \F_r^*$ such that $1/(b+1)=r+s/b$, which implies that $1,b,b^2$ are linearly dependent over $\F_r$, violating the assumption at the beginning of the proof.
\end{proof}

Next, we compute the number of extremal Peisert-type graphs.

\begin{corollary}\label{cor:numberofgraphsCube}
Let $q=r^3$, where $r \geq 3$ is a non-square. Then the number of extremal Peisert-type graphs defined over $\F_{q^2}$ is given by $r(r^5+r^4+r^3+r^2+r+1)$.
\end{corollary}

\begin{proof}
We count the total number of non-canonical cliques containing 0 among all extremal graphs defined over $\F_{q^2}$. According to Proposition~\ref{prop:3}, each such clique is of the form $a\F_r \oplus b\F_r \oplus c\F_r$, where $a, b,c \in \F_{q^2}$ such that $a,b$ are linearly dependent over $\F_q$, and $a,c$ are linearly independent over $\F_q$. Note that the possible number of such triples $(a,b,c)$ is given by $(r^6-1)(r^3-r)(r^6-r^3)$: we can pick $a$ to be an arbitrary element in $\F_{r^6}^*$, then pick $b$ to an arbitrary element in $a\F_{r^3} \setminus a\F_r$, and finally pick $c$ to be an arbitrary element in $\F_{r^6} \setminus a\F_{r^3}$. On the other hand, given such a clique $V$, a similar argument shows that the number of triples $(a,b,c)$ so that $V=a\F_r \oplus b\F_r \oplus c\F_r$ is given by $(r^2-1)(r^2-r)(r^3-r^2)$. Thus, the total number of such cliques is given by:
\begin{align*}
\frac{(r^6-1)(r^3-r)(r^6-r^3)}{(r^2-1)(r^2-r)(r^3-r^2)}
=\frac{r(r^6-1)(r^2-1)(r^3-1)}{(r^2-1)(r-1)^2}
=\frac{r(r^6-1)(r^2+r+1)}{(r-1)}.
\end{align*}
Note that Proposition~\ref{prop:3} also states that a non-canonical clique in an extremal graph uniquely determines the connection set of the graph, and thus determines the graph. By Proposition~\ref{prop:NonCanonicalCliquesCube}, each extremal graph has $(r^2+r+1)$ non-canonical cliques. Thus, the corollary follows immediately.
\end{proof}

\section{Peisert-type graphs induced by geometric objects}\label{sec:induced}
In the previous sections, we discussed special Peisert-type graphs induced by subspaces. In this section, we consider a few families of Peisert-type graphs that are induced by special geometric objects and explore their connections.

Let $\beta \in \F_{q^2} \setminus \F_q$ so that $\{1,\beta\}$ is basis of $\F_{q^2}$ over $\F_q$. For the ease of notations, we identify $\F_{q^2}$ with the 2-dimensional space $V(2,q)$ over $\F_q$ by identifying $x+y\beta$ with $\begin{pmatrix}
x \\
y
\end{pmatrix}$.

\subsection{Graphs $Y_{q,n}(U)$ induced by hyperplanes}\label{sec:induced1}
Let $q=r^n$, where $r$ is a prime power and $n$ is prime. Consider 
$\F_{r^n}$ as an $n$-space over $\F_r$ and let $U$ be an additive coset of an $n-1$ dimensional subspace of $\F_{r^n}$ (equivalently, let $U$ be a hyperplane in $\AG(n,r)$). Note that $|U| = r^{n-1}$.
Let 
$$S(U) = \mathbb{F}_q^* \cup \bigcup\limits_{u \in U} (u+\beta)\mathbb{F}_q^*.$$ 
Let $Y_{q,n}(U)$ be the Peisert-type graph of type $(r^{n-1}+1,r^n)$ defined by the connection set $S(U)$. Further, put 
$$\mathcal{S}_q = \{S(U)~|~ U \text{~is a hyperplane in~} \AG(n,r)\}.$$ 

For any $a \in \mathbb{F}_q^*, b \in \mathbb{F}_q$, let $\varphi_{a,b}: \mathbb{F}_q \to \mathbb{F}_q$ be the affine 
mapping defined by the rule $x \mapsto ax+b$.
Consider the groups 
$$
H = \{\varphi_{a,b}~|~ a \in \mathbb{F}_q^*, b \in \mathbb{F}_q\},
$$
$$
H_0 = \{\varphi_{a,0}~|~ a \in \mathbb{F}_q^*\},
$$
and
$$
H_1  = \{\varphi_{1,b}~|~ b \in \mathbb{F}_q\}.
$$

\begin{proposition}\label{H0IsTransitive}
The group $H_0$ is a group of automorphisms of $\PG_{n-2}(n-1,r)$ acting transitively on the points and the blocks.
\end{proposition}
\begin{proof}
Each element of $H_0$ considered as a non-degenerate matrix from $\GL(n,r)$ can be viewed as an automorphism of $\PG_{n-2}(n-1,r)$. Note that $H_0$ acts transitively on the set of nonzero elements of $\mathbb{F}_q$, which is an $n$-dimensional vector space over $\mathbb{F}_r$. This implies that $H_0$ acts transitively on the points of $\PG_{n-2}(n-1,r)$. By Remark \ref{NontrivialSymmetricDesign} and Lemma \ref{Orbitlemma}, the group $H_0$ also acts transitively on the set of blocks of $\PG_{n-2}(n-1,r)$ (on the set of hyperplanes in $\PG(n-1,r)$).   
\end{proof}

\begin{corollary}\label{HIsTransitiveOnHyperplanes}
The group $H$ acts transitively on the set of hyperplanes of $\AG(n,r)$.
\end{corollary}
\begin{proof}
It follows from two facts.
First, the subgroup $H_1$ stabilizes each class of parallel hyperplanes in $\AG(n,r)$ and acts transitively on the set of hyperplanes of each class. Second, by Proposition \ref{H0IsTransitive}, the subgroup $H_0$ acts transitively on the set of hyperplanes of $\AG(n,r)$ containing zero vector and thus acts transitively on the set of classes of parallel hyperplanes.      
\end{proof}

Note that the connection set $S(U)$ of $Y_{q,n}(U)$ can be written as follows:
$$
S(U) = 
\begin{pmatrix}
1 \\
0
\end{pmatrix} 
\mathbb{F}_q^*
\cup
\bigcup\limits_{u \in U}
\begin{pmatrix}
u \\
1
\end{pmatrix} \mathbb{F}_q^*.
$$

Let $A \in \GL(2,q)$ be a matrix and $v \in V(2,q)$ be a vector. Let 
$$
\varphi_{A,v}: V(2,q) \to V(2,q)
$$ 
be a mapping defined by the following rule: for any $u \in V(2,q)$,
$$
\varphi_{A,v}(u) = Au+v.
$$
It is clear that for any $A \in \GL(2,q)$ and $v \in V(2,q)$, the mapping  $\varphi_{A,v}$ is a collineation (automorphism) of $\AG(2,q)$. Let $$G=\left\{\varphi_{A,v} ~|~ A \in \GL(2,q), v \in V(2,q)\right\}.$$
For any $a \in \mathbb{F}_{q}^*$, $b\in \mathbb{F}_q$, let
$\psi_{a,b}:\mathbb{F}_{q^2} \to \mathbb{F}_{q^2}$ be the mapping defined by the following rule: 
$$
\psi_{a,b}\bigg(
\begin{pmatrix}
\gamma_1 \\
\gamma_2
\end{pmatrix}
\bigg)
=
\begin{pmatrix}
a & b\\
0 & 1
\end{pmatrix}
\begin{pmatrix}
\gamma_1 \\
\gamma_2
\end{pmatrix}
=
\begin{pmatrix}
a\gamma_1+b\gamma_1\\
\gamma_2 
\end{pmatrix}.
$$
Let $$L = \{\psi_{a,b}~|~a \in \mathbb{F}_q^*, b \in \mathbb{F}_q\}.$$ Note that $L$ is a subgroup in $G$. Moreover, $L$ is a subgroup of the stabilizer (in $G$) of the zero vector.

\begin{lemma}\label{LIsTransitiveOnScal}
The group $L$ preserves the set $\mathcal{S}_q$ and acts on it transitively.    
\end{lemma}
\begin{proof}
Let $\psi_{a,b} \in L$ and let $U$ be a hyperplane in $\AG(n,r)$. 
Then
$$
\psi_{a,b}(S(U))=\begin{pmatrix}
1 \\
0
\end{pmatrix} 
\mathbb{F}_q^*
\cup
\bigcup\limits_{u \in U}
\begin{pmatrix}
au+b \\
1
\end{pmatrix} \mathbb{F}_q^* =S(aU+b).
$$
Thus, the group $L$ preserves the set $\mathcal{S}_q$. The transitivity of the action then follows from Corollary \ref{HIsTransitiveOnHyperplanes}.
\end{proof}

\begin{corollary}\label{hyperplaneinduceiso}
For any two hyperplanes $U_1,U_2$ in $\AG(n,r)$, the graphs $Y_{q,n}(U_1)$ and $Y_{q,n}(U_2)$ are isomorphic.    
\end{corollary}
\begin{proof}
It follows from Lemma~\ref{lem:Cayley_iso}, Lemma \ref{LIsTransitiveOnScal}, and the fact that $L$ stabilizes the zero vector. 
\end{proof}

\begin{corollary}\label{cor:iso3}
Let $q=r^3$, where $r \geq 3$ is a non-square. Then all extremal Peisert-type graphs defined over $\F_{q^2}$ are isomorphic.
\end{corollary}

\begin{proof}
Let $X$ be an extremal graph. In view of Corollary~\ref{hyperplaneinduceiso}, it suffices to show that $X$ is isomorphic to the graph $Y_{q,3}(U)$ for some hyperplane $U$ in $\AG(3,r)$. By Proposition~\ref{prop:3}, there exist $a,b,c \in \F_{q^2}$, such that $S=(a \F_{r} \oplus b \F_{r} \oplus c \F_{r} \setminus \{0\})\F_q^*$, where $a, b,c \in \F_{q^2}$ such that $a,b$ are linearly dependent over $\F_q$, and $a,c$ are linearly independent over $\F_q$. Let $\beta=c/a$ and $U=\F_r \oplus \frac{b}{a}\F_r$, and let 
\begin{equation}
S'=\F_q^* \sqcup \bigsqcup_{u \in U} (u+\beta) \F_q^*.
\end{equation}
Then we have $S=aS'$. Let $X'$ be the Peisert-type graph with connection set $S'$. Then $X$ is isomorphic to $X'$ by Lemma~\ref{lem:Cayley_iso}. Note that $X'=Y_{q,3}(U)$ and thus the proof is complete.  
\end{proof}

Now we are ready to present the proof of Theorem~\ref{thm:thm3}.
\begin{proof}[Proof of Theorem~\ref{thm:thm3}]
When $r=2$, we have verified the statement of the theorem. In particular, it turned out that all Peisert-type graphs of type $(5,8)$ are extremal.

Next we assume that $r \geq 3$. The theorem follows from Corollary~\ref{cor:numberofgraphsCube}, Corollary~\ref{cor:iso3}, and Proposition \ref{prop:NonCanonicalCliquesCube}.
\end{proof}

When $q$ is a square, a similar argument as in the proof of Corollary~\ref{cor:iso3} shows that all extremal Peisert-type graphs defined over $\F_{q^2}$ are isomorphic; we refer to Proposition~\ref{prop:iso} for a stronger statement. In general, extremal Peisert-type graphs of the same order may not be unique. The following example gives two non-isomorphic extremal graphs over $\F_{q^2}$, where $q=32$.

\begin{example}\label{ex:q=32}
Let $q=2^5$. Let $\varepsilon$ be a primitive element in $\mathbb{F}_{q}$ such that $\varepsilon$ is a root of the irreducible polynomial $t^5+t^2+1 \in \mathbb{F}_2[t]$. Let $\beta$ be a root of the irreducible polynomial $t^2+t+1 \in \mathbb{F}_2[t]$; note that $\beta \in \F_{q^2} \setminus \F_q$. Consider the Peisert-type graph $X_1$ induced by the $\F_2$-subspace $V_1$ generated by the elements $\{ 1, \varepsilon, \beta, \varepsilon^{16}\beta, \varepsilon^{21} + \varepsilon^{9}\beta\}$, that is, $X_1=\operatorname{Cay}(\F_{q^2}^+, V_1\F_q \setminus \{0\})$. Similarly, consider the Peisert-type graph $X_2$ induced by the $\F_2$-subspace $V_2$ generated by the elements $\{ 1, \varepsilon,  \varepsilon^2,  \varepsilon^3, \beta\}$. 

For $i \in \{1,2\}$, from the definition of $X_i$, it is easy to verify that $V_i$ is a non-canonical clique in $X_i$ and $X_i$ is a Peisert-type graph of type $(17, 32)$. Thus, in view of Theorem~\ref{thm:thm1}, $X_1$ and $X_2$ are extremal graphs. We have verified that $X_1$ is not isomorphic to $X_2$, showing that there are at least $2$ non-isomorphic extremal Peisert-type graphs defined over $\F_{q^2}$.
\end{example}

\begin{remark}
Let $q=r^n$, where $n$ is a prime, and let $U$ be a hyperplane in $\AG(n,r)$. Similar to the proof of Corollary~\ref{cor:iso3}, one can verify that $Y_{q,n}(U)$ is a Peisert-type graph of type $(r^{n-1}+1,r^n)$ without the strict-EKR property. Note that there might be multiple ways to write $q=r^n$, in particular, if $n$ is the smallest such prime, then $Y_{q,n}(U)$ is an extremal Peisert-type graph in view of Theorem~\ref{thm:thm1}.
\end{remark}

\subsection{Graphs $X_q$ induced by special ovals}\label{X_q}
Let $q = r^2$. Note that $\mathbb{F}_r$ is a hyperplane (a line) in $\AG(2,r)$. Consider the extremal Peisert-type graph  $Y_{q,2}(\mathbb{F}_r)$ of type $(r+1,q)$ (that is, we put $U = \mathbb{F}_r$ in the definition of $Y_{q,2}(U)$). 

We define a new Peisert-type graph $X_q$ induced by a special oval. Consider the oval $Q := \{\gamma \in \mathbb{F}_q^*~|~ \gamma^{r+1} = 1\}$. Note that $Q$ played an important role in the construction of maximal cliques in Paley graphs and generalized Paley graphs \cite{GKSV18, GSY23}. Let $S = \bigcup\limits_{\delta \in Q} (\delta+\beta)\mathbb{F}_q^*$. Let $X_q$ be the Peisert-type graph of type $(\sqrt{q}+1,q)$ defined by the connection set $S$. 

\begin{proposition}\label{YqXq_iso}
The graphs $Y_{q,2}(\mathbb{F}_r)$ and $X_q$ are isomorphic.   
\end{proposition}
\begin{proof}
\medskip
Note that the connection set of $X_q$ can be written as
$$
S = \bigcup\limits_{\delta \in Q} 
\begin{pmatrix}
\delta\\
1 
\end{pmatrix}\mathbb{F}_q^*.
$$
On the other hand, the connection set $S(\mathbb{F}_r)$ of $Y_{q,2}(\mathbb{F}_r)$ can be written as
$$
S(\mathbb{F}_r) = 
\begin{pmatrix}
1 \\
0
\end{pmatrix} 
\mathbb{F}_q^*
\cup
\bigcup\limits_{u \in \mathbb{F}_r}
\begin{pmatrix}
u \\
1
\end{pmatrix} \mathbb{F}_q^*.
$$

Let $\gamma \in \mathbb{F}_{q}$ such that $\gamma^r \ne \gamma$. Then the matrix
$$
A =
\begin{pmatrix}
\gamma^r & \gamma\\
1 & 1
\end{pmatrix}
$$
is non-singular. Consider the collineation $\varphi_{A,0}$ and denote it by $\varphi_A$ for short. In view of the finiteness of the set $S$ and Lemma~\ref{lem:Cayley_iso}, it suffices to show that
$\varphi_A\left(
\begin{pmatrix}
\delta \\
1
\end{pmatrix}
\right) \in S(\mathbb{F}_r)$ for any $\delta \in Q$. 
We have
$$
\begin{pmatrix}
\gamma^r & \gamma\\
1 & 1
\end{pmatrix}
\begin{pmatrix}
\delta \\
1
\end{pmatrix}
=
\begin{pmatrix}
\gamma+\delta\gamma^r \\
1+\delta
\end{pmatrix}.
$$

Consider the following two cases.

\textbf{Case 1:} $\delta = - 1$. Then
$$
\begin{pmatrix}
\gamma+\delta\gamma^r \\
1+\delta
\end{pmatrix}
=
\begin{pmatrix}
\gamma-\gamma^r \\
0
\end{pmatrix}.
$$
Since $\gamma-\gamma^r \ne 0$, the image lies in $S(\mathbb{F}_r)$.

\medskip
\textbf{Case 2:} $\delta \ne - 1$. Note that
$$
\begin{pmatrix}
\gamma+\delta\gamma^r \\
1+\delta
\end{pmatrix} \in S(\mathbb{F}_r) \text{ if and only if } 
\begin{pmatrix}
\frac{\gamma+\delta\gamma^r}{1+\delta} \\
1
\end{pmatrix} \in S(\mathbb{F}_r),$$
and it suffices to show that
$\frac{\gamma+\delta\gamma^r}{1+\delta}$ belongs to $\mathbb{F}_{r}$. To do this, we show that
$$
\left(\frac{\gamma+\delta\gamma^r}{1+\delta}\right)^r 
=
\frac{\gamma+\delta\gamma^r}{1+\delta}.
$$
Indeed, we have
$$
\left(\frac{\gamma+\delta\gamma^r}{1+\delta}\right)^r = 
\frac{\gamma^r+\delta^r\gamma}{1+\delta^r}.
$$
Since $\delta \in Q$, we have $\delta^{r+1} = 1$ and, consequently,
$\delta^r = \frac{1}{\delta}$. Thus,
$$
\frac{\gamma^r+\delta^r\gamma}{1+\delta^r} = 
\frac{\gamma^r+\frac{1}{\delta}\gamma}{1+\frac{1}{\delta}} =
\frac{\delta\gamma^r+\gamma}{\delta+1}
$$
holds, as required.
\end{proof}

\section{An isomorphism between $X_q$ and the affine polar graph $VO^+(4,\sqrt{q})$}\label{sec:polar}

Let $q = r^2$. Further, let $g(t) \in \mathbb{F}_r[t]$ be a quadratic irreducible polynomial and $\alpha$ be its root. We have $\mathbb{F}_q = \{x+y\alpha~|~ x,y \in \mathbb{F}_r\}$, and we naturally view $\mathbb{F}_q$ as a $2$-dimensional vector space over $\mathbb{F}_r$ as follows:
for any $x,y \in \mathbb{F}_r$, we identify the element $x+y\alpha \in \mathbb{F}_q$ and the vector $(x,y) \in V(2,r)$.

The graph $X_q$ was defined on the vectors of $V(2,q)$ in Section~\ref{X_q}. We identify the vector $(x_1+y_1\alpha,x_2+y_2\alpha) \in V(2,q)$ and the vector $(x_1,y_1,x_2,y_2) \in V(4,r)$. In Sections \ref{subsecOddq} and \ref{subsecEvenq}, we redefine $X_q$ on the vectors of $V(4,r)$.

\subsection{Odd $q$}\label{subsecOddq}
In this section, we consider the case when $q$ is odd. Let $d$ be a non-square element in $\mathbb{F}_r^*$. Then we may assume that $g(t) = t^2 - d$ and $\alpha^2 = d$.

Recall that the connection set $S$ of $X_q$ is given by
$$
S = \{(\gamma_1,\gamma_2)~|~ \gamma_1,\gamma_2 \in \mathbb{F}_q, \gamma_1 \ne 0, \gamma_2 \ne 0, \gamma_1^{r+1} = \gamma_2^{r+1}\}.
$$
Taking into account the identification of elements $V(2,q)$ and $V(4,r)$, we rewrite $S$ as follows:
$$
S = \{(x_1,y_1,x_2,y_2)~|~ x_1^2-y_1^2d = x_2^2-y_2^2d\},
$$
where $x_1,y_1,x_2,y_2$ run over $\mathbb{F}_r$ and 
$(x_1,y_1) \ne (0,0), (x_2,y_2) \ne (0,0)$. 
Or, equivalently,
$$
S = \{(x_1,y_1,x_2,y_2)~|~ x_1^2-y_1^2d - x_2^2+y_2^2d = 0\},
$$
where $x_1,y_1,x_2,y_2$ run over $\mathbb{F}_r$ and 
$(x_1,y_1) \ne (0,0), (x_2,y_2) \ne (0,0)$. 
Thus, the graph $X_q$ can be viewed as the graph $Y_{f_1}$ associated with quadratic form 
$$f_1(x_1',y_1',x_2',y_2') = (x_1')^2-(y_1')^2d - (x_2')^2+(y_2')^2d.$$

On the other hand, the graph $VO^+(4,r)$ is also defined on the vectors of $V(4,r)$ with connection set 
$$
\{(x_1,y_1,x_2,y_2)~|~ x_1y_1+x_2y_2 = 0\},
$$
where $x_1,y_1,x_2,y_2$ run over $\mathbb{F}_r$.
Thus, the graph $VO^+(4,r)$ can be viewed as the graph $Y_{f_2}$ associated with quadratic form 
$$f_2(x_1,y_1,x_2,y_2) = x_1y_1 + x_2y_2.$$

\begin{proposition}\label{IsomorphismOddr}
    If $r$ is an odd prime power, then the graphs $X_{r^2}$ and $VO^+(4,r)$ are isomorphic.
\end{proposition}
\begin{proof}
In view of Lemma \ref{EquivalentFormsGiveIsomorphicGraphs}, it suffices to show that the forms $f_1$ and $f_2$ are equivalent.
Consider the change of variables
$$
\begin{cases}
x_1 = x_1' + x_2',\\
y_1 = x_1' - x_2',\\
x_2 = dy_1' + dy_2',\\
y_2 = -y_1' + y_2'.
\end{cases}
$$
given by the matrix
$$
B =
\begin{pmatrix}
1 & 0 & 1 & 0\\
1 & 0 & -1 & 0\\
0 & d & 0 & d\\
0 & -1 & 0 & 1
\end{pmatrix}.
$$
It is easy to see that $det(B) = 4d \ne 0$ and 
$$
\underbrace{x_1y_1+x_2y_2}_{f_1} = \underbrace{(x_1')^2-(y_1')^2d - (x_2')^2+(y_2')^2d}_{f_2},
$$
which completes the proof. 
\end{proof}

\subsection{Even $q$}\label{subsecEvenq}

Let $r = 2^s$, where $s$ is a positive integer, and $q = r^2$.  Recall the {\em absolute trace map} $\Tr_{\F_r}:\F_r \to \F_2$ is defined to be
$$
\Tr_{\F_r}(x)=\alpha+\alpha^2+\cdots+\alpha^{2^{s-1}}, \quad \forall x \in \F_r.
$$
Choose any $d \in \mathbb{F}_r$ such that $\Tr_{\F_r}(d) = 1$. Then the polynomial $g(t) = t^2 + t + d \in \mathbb{F}_r[t]$ is irreducible (see \cite[Corollary 3.79]{LN97}). Let $\alpha$ be a root of $g(t)$, that is, $\alpha^2 = \alpha + d$. Then we have $\mathbb{F}_q = \{x+y\alpha~|~ x,y \in \mathbb{F}_r\}$.

Let $N: \mathbb{F}_q^* \to \mathbb{F}_r^*$ be the norm mapping from $\mathbb{F}_q^*$ onto $\mathbb{F}_r^*$ defined by the following rule:
$$
\text{for any~} \gamma \in \mathbb{F}_q^*, \quad N(\gamma) = \gamma^{r+1}.
$$

\begin{lemma}\label{NormEvenq}
The following statements hold.\\
{\rm (1)} For any $x+y\alpha \in \mathbb{F}_q^*$ and for any positive integer k, the equality
$$
(x+y\alpha)^{2^k} = x^{2^k} + y^{2^k}\bigg(\sum\limits_{j = 0}^{k-1}d^{2^j}\bigg) + y^{2^k}\alpha.
$$
holds.\\
{\rm (2)} For any $x+y\alpha \in \mathbb{F}_q^*$, the equality
$$(x+y\alpha)^r = x + y +y\alpha$$
holds.\\
{\rm (3)} For any $x+y\alpha \in \mathbb{F}_q^*$, the equality
$$N(x+y\alpha) = x^2 + xy +y^2d$$
holds.
\end{lemma}
\begin{proof}
(1) Note that for each $z,w \in \F_r$, we have $(z+w\alpha)^2=(z^2 + w^2d) + w^2\alpha$. It follows that 
$$
(x+y\alpha)^4 = \big((x^2 + y^2d) + y^2\alpha\big)^2=x^4+y^4d^2+y^4d +y^4 \alpha =\big(x^4+y^4(d+d^2)\big)+y^4\alpha. 
$$
By repeating the same process, the required formula can be proved easily by induction.

(2) Applying item (1) in the case $k = s$, we get
$$
(x+y\alpha)^r = x^r + y^r\bigg(\sum\limits_{j = 0}^{s-1}d^{2^j}\bigg) + y^r\alpha=x^r + y^r\Tr_{\F_r}(d) + y^r\alpha.
$$
Taking into account that $\Tr_{\F_r}(d) = 1$, $x^r = x$ and $y^r = y$, we get
$$
(x+y\alpha)^r = x + y +y\alpha.
$$

(3) In view of item (2), we have
$$
N(x+y\alpha) = (x+y\alpha)^{r+1} = (x+y\alpha)^r(x+y\alpha) = (x+y+y\alpha)(x+y\alpha) = x^2+xy+y^2d.
$$

\end{proof}

Recall that the connection set $S$ of $X_q$ is given by
$$
S = \{(\gamma_1,\gamma_2)~|~ \gamma_1,\gamma_2 \in \mathbb{F}_q, \gamma_1 \ne 0, \gamma_2 \ne 0, \gamma_1^{r+1} = \gamma_2^{r+1}\}.
$$
Taking into account the identification of elements $V(2,q)$ and $V(4,r)$, using Lemma \ref{NormEvenq}(3), we rewrite $S$ as follows:
$$
S = \{(x_1,y_1,x_2,y_2)~|~ x_1^2 + x_1y_1 + y_1^2d = x_2^2 + x_2y_2 + y_2^2d\},
$$
where $x_1,y_1,x_2,y_2$ run over $\mathbb{F}_r$ and 
$(x_1,y_1) \ne (0,0), (x_2,y_2) \ne (0,0)$. 
Or, equivalently,
$$
S = \{(x_1,y_1,x_2,y_2)~|~ x_1^2 + x_1y_1 + y_1^2d + x_2^2 + x_2y_2 + y_2^2d = 0\},
$$
where $x_1,y_1,x_2,y_2$ run over $\mathbb{F}_r$ and 
$(x_1,y_1) \ne (0,0), (x_2,y_2) \ne (0,0)$. 
Thus, the graph $X_q$ can be viewed as the graph $Y_{f_1}$ associated with quadratic form 
$$f_1(x_1',y_1',x_2',y_2') = (x_1')^2 + x_1'y_1' + (y_1')^2d + (x_2')^2 + x_2'y_2' + (y_2')^2d.$$

On the other hand, the graph $VO^+(4,r)$ is also defined on the vectors of $V(4,r)$ with connection set 
$$
\{(x_1,y_1,x_2,y_2)~|~ x_1y_1+x_2y_2 = 0\},
$$
where $x_1,y_1,x_2,y_2$ run over $\mathbb{F}_r$.
Thus, the graph $VO^+(4,r)$ can be viewed as the graph $Y_{f_2}$ associated with quadratic form 
$$f_2(x_1,y_1,x_2,y_2) = x_1y_1 + x_2y_2.$$

\begin{proposition}\label{IsomorphismEvenr}
    If $r$ is a power of 2, then the graphs $X_{r^2}$ and $VO^+(4,r)$ are isomorphic.
\end{proposition}
\begin{proof}
In view of Lemma \ref{EquivalentFormsGiveIsomorphicGraphs}, it suffices to show that the forms $f_1$ and $f_2$ are equivalent.
Consider the change of variables
$$
\begin{cases}
x_1 = x_1' + y_1'+ x_2',\\
y_1 = x_1' + x_2',\\
x_2 = y_1' + y_2',\\
y_2 = dy_1' + x_2' + dy_2'.
\end{cases}
$$
given by the matrix
$$
B =
\begin{pmatrix}
1 & 1 & 1 & 0\\
1 & 0 & 1 & 0\\
0 & 1 & 0 & 1\\
0 & d & 1 & d
\end{pmatrix}.
$$
It is easy to see that $det(B) = 1 \ne 0$ and 
$$
\underbrace{x_1y_1+x_2y_2}_{f_1} = \underbrace{(x_1')^2 + x_1'y_1' + (y_1')^2d + (x_2')^2 + x_2'y_2' + (y_2')^2d}_{f_2},
$$
which completes the proof. 
\end{proof}

\subsection{Implications to extremal Peisert-type graphs}

\begin{proposition}\label{prop:iso}
Let $q$ be a square. If $X$ is an extremal Peisert-type graph of type $(\sqrt{q}+1,q)$, then 
$X$ is isomorphic to $X_q$ and the affine polar graph $VO^+(4,\sqrt{q})$.
\end{proposition}
\begin{proof}
Let $S$ be the connection set of $X$. 
By Proposition~\ref{prop:2}, $S=(a \F_{\sqrt{q}} \oplus b \F_{\sqrt{q}})\F_q \setminus \{0\}$, where $a, b \in \F_{q^2}$ are linearly independent over $\F_q$. It follows that 
$$
S=b \F_q^* \cup \bigcup_{\lambda \in \F_{\sqrt{q}}} (a+\lambda b)\F_q^*
$$
Let $\beta=a/b$; then $\beta \in \F_{q^2} \setminus \F_q$. Consider
$$
S'=\F_q^* \cup \bigcup_{\lambda \in \F_{\sqrt{q}}} (\beta+\lambda)\F_q^*
$$
so that $bS'=S$. Let $X'$ be the Peisert-type graph with connection set $S'$. Then $X$ is isomorphic to $X'$ by Lemma~\ref{lem:Cayley_iso}. Note that $X'=Y_{q,2}(\F_{\sqrt{q}})$ and thus Proposition~\ref{YqXq_iso} implies that $X'$ is isomorphic to $X_q$.  Finally, Proposition~\ref{IsomorphismOddr} and Proposition~\ref{IsomorphismEvenr} imply that $X_q$ is isomorphic to $VO^+(4,\sqrt{q})$. We conclude that $X$ is isomorphic to $VO^+(4,\sqrt{q})$.
\end{proof}

\begin{corollary}\label{cor:noHM}
Let $q$ be a square. If $X$ is an extremal Peisert-type graph of type $(\sqrt{q}+1,q)$, then 
there is no Hilton-Milner type result: all maximal cliques in $X$ all maximum.
\end{corollary}
\begin{proof}
We have shown in Proposition~\ref{prop:iso} that $X$ is isomorphic to the affine polar graph $VO^+(4,\sqrt{q})$. By Proposition~\ref{prop:NonCanonicalCliques}, the number of non-canonical cliques in $X$ that contains $0$ is $\sqrt{q}+1$. On the other hand, it is clear that the number of canonical maximum cliques in $X$ that contains $0$ is $\sqrt{q}+1$. Thus, the number of maximum cliques containing $0$ is $2(\sqrt{q}+1)$. However, Lemma~\ref{lem:maximalclique} implies that the number of maximal cliques contains $0$ is $2(\sqrt{q}+1)$. This means all maximal cliques in $X$ are maximum.  
\end{proof}

Rank $3$ graphs are of special interest in algebraic graph theory; we refer to the book \cite{BV22} for extensive discussions. It is known that Paley graphs and Peisert graphs are rank $3$ graphs, since they are self-complementary and symmetric. In particular, the Paley graphs and Peisert graphs which are of Peisert-type are rank $3$ graphs. Interestingly, when $q$ is a square, all extremal Peisert-type graphs defined over $\F_{q^2}$ are rank 3 graphs.

\begin{corollary}\label{XqIsRank3}
Let $q$ be a square. If $X$ is an extremal Peisert-type graph of type $(\sqrt{q}+1,q)$, then $X$ is a rank 3 graph.   
\end{corollary}
\begin{proof}
The corollary follows from Proposition~\ref{prop:iso} and the fact that 
$VO^+(4,\sqrt{q})$ is a rank 3 graph (see \cite[page 98]{BV22}).
\end{proof}

\section{Tightness of the weight-distribution bound for both non-principal eigenvalues of the graphs $X_q$}\label{sec:eigenfunction}

In this section, we show that if $q$ is a square, then the weight-distribution bound is tight for both non-principal eigenvalues of all extremal graphs defined on $\F_{q^2}$. It suffices to consider the graph $X_q$ in view of Proposition~\ref{prop:iso}. Recall that $X_q$ is the Peisert-type graph of type $(\sqrt{q}+1,q)$ defined by the connection set $S$, where 
$$
S = \bigcup\limits_{\delta \in Q} (\delta+\beta)\mathbb{F}_q^*, \quad Q= \{\gamma \in \mathbb{F}_q^*~|~ \gamma^{\sqrt{q}+1} = 1\}, \quad \text{and } \beta \in \F_{q^2} \setminus \F_q.
$$
By Lemma~\ref{SRG}, the two non-principal eigenvalues of $X_q$ are given by $q-\sqrt{q}$ and $-\sqrt{q}-1$. Lemma~\ref{WDBsrg} then gives the weight-distribution bound for eigenfunctions.

Consider the following 2-dimensional $\mathbb{F}_{\sqrt{q}}$-subspace in $\mathbb{F}_{q^2}$:
\begin{align*}
        C_q &= (1+\beta)\mathbb{F}_{\sqrt{q}} + (\varepsilon^{\sqrt{q}}+\varepsilon\beta)\mathbb{F}_{\sqrt{q}} \\
&= \{(1+\beta)a + (\varepsilon^{\sqrt{q}}+\varepsilon\beta)b ~|~ a,b \in \mathbb{F}_{\sqrt{q}}\} \\
&= \{a+b\varepsilon^{\sqrt{q}}+(a+b\varepsilon)\beta  ~|~ a,b \in \mathbb{F}_{\sqrt{q}}\} \\
&= \{(a+b\varepsilon)^{\sqrt{q}}+(a+b\varepsilon)\beta  ~|~ a,b \in \mathbb{F}_{\sqrt{q}}\} \\
&=\{\gamma^{\sqrt{q}}+\gamma\beta  ~|~ \gamma \in \mathbb{F}_{q}\} \\
&=\{\gamma(\gamma^{\sqrt{q}-1}+\beta)  ~|~ \gamma \in \mathbb{F}_{q}\} \subset S \cup \{0\}.
    \end{align*}

\begin{proposition}
The subspace $C_q$ is a non-canonical clique in $X_q$. Moreover, the intersection of any canonical clique in $X_q$ containing 0 and $C_q$ has exactly $\sqrt{q}-1$ nonzero elements.
\end{proposition}
\begin{proof}
    Since $C_q$ is a subspace, it is clear that $C_q$ forms a maximum clique in $X_q$. 
    It follows from $C_q = \{\gamma(\gamma^{\sqrt{q}-1}+\beta)  ~|~ \gamma \in \mathbb{F}_{q}\}$ that the intersection of any canonical clique in $X_q$ and $C_q$ has exactly $\sqrt{q}-1$ nonzero elements (these elements are given by the elements $\gamma \in \mathbb{F}_q^*$ lying in the same coset of $\mathbb{F}_{\sqrt{q}}^*$ in $\mathbb{F}_{q}^*$), which means that the clique $C_q$ is indeed non-canonical.
\end{proof}

\begin{corollary}\label{espilonCq}
    For any $i \in \{0,1,\ldots,\sqrt{q}\}$, the set $\varepsilon^iC_q$ forms a non-canonical clique in $X_q$. Moreover, for any $i,j \in \{0,1,\ldots,\sqrt{q}\}$ such that $i \ne j$, we have $\varepsilon^iC_q \cap \varepsilon^jC_q = \{0\}$.
\end{corollary}
\begin{proof}
This was already shown in the proof of Proposition~\ref{prop:NonCanonicalCliques}, but let us provide a different look. It follows from three facts. First, multiplication by an element from $\mathbb{F}_q^*$ is an automorphism of $X_q$. Second, $\{1,\varepsilon, \varepsilon^2, \ldots, \varepsilon^{\sqrt{q}}\}$ is a system of representatives of the cosets of $\mathbb{F}_{\sqrt{q}}^*$ in $\mathbb{F}_{q}^*$. Third, the intersection of any canonical clique in $X_q$ and $C_q$ has exactly $\sqrt{q}-1$ nonzero elements, and these elements are given by the elements $\gamma \in \mathbb{F}_q^*$ lying in the same coset of $\mathbb{F}_{\sqrt{q}}^*$ in $\mathbb{F}_{q}^*$.
\end{proof}

Note that Delsarte cliques exist in both $X_q$ and its complement. 
In view of Corollary \ref{XqIsRank3}, any edge of $X_q$ lies in a constant number of Delsarte cliques, and any edge of the complement of $X_q$ does so as well. Thus, in view of 
Lemma \ref{WDBsrg1}(2) (resp. Lemma \ref{WDBsrg1}(1)), any induced pair of isolated cliques of size $q-\sqrt{q}$ in $X_q$ (resp. any induced complete bipartite subgraph with parts of size $\sqrt{q}+1$ in $X_q$) gives an optimal eigenfunction whose cardinality of support meets the weight-distribution bound. For the sake of completeness, we also explicitly construct optimal eigenfunctions below.

Let $C$ be a canonical clique of $X_q$ containing 0 and let $\varepsilon^i C_q$, where $i \in \{0,1,\ldots,\sqrt{q}\}$ be a non-canonical clique of $X_q$ described in Corollary \ref{espilonCq}. Let $D = C \cap \varepsilon^i C_q$. Note that $|D| = \sqrt{q}$. Define a function $f_1: \F_{q^2} \to \mathbb{R}$ by the following rule:
$$
f_1(\gamma) = 
\left\{
  \begin{array}{cc}
    1, & \gamma \in C \setminus D; \\
    -1, & \gamma \in \varepsilon^i C_q \setminus D; \\
    0, & \gamma \not \in C \cup \varepsilon^i C_q.
  \end{array}
  \right.
$$

\begin{proposition}\label{WDBrXqisTight1}
    The function $f_1$ is a $(q-\sqrt{q})$-eigenfunction of $X_q$ whose cardinality of support is $2(q-\sqrt{q})$. Thus, the weight-distribution bound is tight for the positive non-principal eigenvalue $q-\sqrt{q}-1$ of $X_q$.     

\end{proposition}
\begin{proof}
    It follows from the definition of $f$ and the fact that $C$ and $C_q$ are Delsarte (regular) cliques with nexus $\sqrt{q}$ that $(C \setminus D) \cup (C_q \setminus D)$ induce a pair of isolated cliques of size $q-\sqrt{q}$.
    In view of Lemma~\ref{WDBsrg1}(2) and Corollary \ref{XqIsRank3}, the weight-distribution bound is tight for the eigenvalue $q-\sqrt{q}-1$.
\end{proof}

Let $T_0 = Q$ and $T_1 = Q\beta$. Note that $T_0$ and $T_1$ are subsets of the lines with slopes $1$ and $\beta$ in $\AG(2,q)$. These lines do not intersect with $S$ and thus are cocliques in $X_q$, which means that $T_0$ and $T_1$ are cocliques. 

Let $\gamma_2 \in T_0$ and $\gamma_1\beta \in T_1$ be two arbitrary elements from the cocliques $T_0$ and $T_1$. Consider their difference and take into account that $Q$ is a subgroup of order $\sqrt{q}+1$ in $\mathbb{F}_q^*$ and $-Q = Q$:
$$
\gamma_2 - \gamma_1\beta = \gamma_2+\gamma_1'\beta= \gamma_1'(\gamma_2(\gamma_1')^{-1}+\beta) = \gamma_1'(\gamma_2'+\beta) \in S,
$$
where $\gamma_1',\gamma_2'$ are uniquely determined elements from $Q$. This means that $T_0 \cup T_1$ induces a complete bipartite subgraph in $X_q$ with parts $T_0$ and $T_1$ of size $\sqrt{q}+1$.

Define a function $f_2: \F_{q^2} \to \mathbb{R}$ by the following rule:
$$
f_2(\gamma) = 
\left\{
  \begin{array}{cc}
    1, & \gamma \in T_0; \\
    -1, & \gamma \in T_1; \\
    0, & \gamma \not \in T_0 \cup T_1.
  \end{array}
  \right.
$$

\begin{proposition}\label{WDBrXqisTight2}
    The function $f_2$ is a $(-\sqrt{q}-1)$-eigenfunction of $X_q$ whose cardinality of support is $2(\sqrt{q}+1)$. Thus, the weight-distribution bound is tight for the negative non-principal eigenvalue $-\sqrt{q}-1$ of $X_q$.     
\end{proposition}
\begin{proof}
    Since $T_0 \cup T_1$ induces a complete bipartite subgraph with parts $T_0$ and $T_1$ of size $\sqrt{q}+1$, the result follows from Lemma~\ref{WDBsrg1}(1). For the sake of completeness, we verify that $f_2$ is indeed an eigenfunction, that is, for each vertex outside of $T_0 \cup T_1$, the number of its neighbors in $T_0$ and $T_1$ is the same. Indeed, this follows from the geometric structure of $X_q$. If a vertex $\gamma \notin T_0 \cup T_1$ is adjacent to a vertex $\gamma_1 \in T_i$, then, because of the structure of the subgraph induced by $T_0 \cup T_1$, the adjacency line connecting $\gamma$ and $\gamma_1$ intersects each of the sets $T_0$ and $T_1$ in exactly one point.  
\end{proof}

The following corollary follows from Proposition~\ref{WDBrXqisTight1} and Proposition~\ref{WDBrXqisTight2}.

\begin{corollary}\label{cor:WDB}
Let $q$ be a square. If $X$ is an extremal Peisert-type graph defined on $\F_{q^2}$, then 
the weight-distribution bound is tight for both non-principal eigenvalues of $X$.
\end{corollary}

Now we are ready to prove Theorem~\ref{thm:thm2}.

\begin{proof}
When $q=4$, we have verified the statement of the theorem. In particular, it turned out that all Peisert-type graphs of type $(3,4)$ are extremal.

Next we assume that $q \geq 9$. The theorem follows immediately from Corollary~\ref{cor:numberofgraphs}, Corollary~\ref{cor:3induce}, Proposition~\ref{prop:iso}, Proposition~\ref{prop:NonCanonicalCliques}, Lemma~\ref{lem:equivalent}, Corollary~\ref{cor:noHM}, and Corollary~\ref{cor:WDB}. 
\end{proof}

\section{Implications to the study of extremal block graphs of orthogonal arrays with non-canonical cliques}\label{sec:OA}

In \cite{AGLY22}, the fact that Peisert-type graphs can be realized as block graphs of orthogonal arrays was explored, and it was shown that the definitions of canonical cliques previously given for Peisert-type graphs and block graphs of orthogonal arrays agree with each other. In this paper we defined extremal Peisert-type graphs having non-canonical cliques. In fact, this definition can be naturally extended to the class of block graphs of orthogonal arrays obtained from affine planes different from $\AG(2,q)$ and having non-canonical cliques (for the background on the block graphs of orthogonal arrays, see \cite[Section 5.5]{GM15}).
It is well known (for example, see \cite[Theorem 10.4.1]{GR01}) that an orthogonal array $OA(m,n)$ is equivalent to $m-2$ mutually orthogonal Latin squares. Further, it is well known (for example, see \cite[Theorem 5.1.2]{KD15}) that no more than $n-1$ mutually orthogonal Latin squares of order
$n$ exist. This implies that for an orthogonal array $OA(m,n)$, we necessarily have $m \le n + 1$. A set of $n-1$ mutually orthogonal Latin squares (an orthogonal array $OA(n+1,n)$) is called \emph{complete} (see \cite[p. 161]{KD15}). It is well known (for example, see \cite[Theorem 5.2.2, Theorem 8.1.1]{KD15}) that the existence of a complete set of mutually orthogonal Latin squares of order $n$ (equivalently, a complete orthogonal array $OA(n+1,n)$) is equivalent to the existence of a projective plane of order $n$, whose existence is known to be equivalent to the existence of an affine plane of order $n$. Let $A$ be a complete orthogonal array of type $OA(n+1,n)$ and let $A_1$ be an orthogonal array $OA(m,n)$ obtained as a subset of rows of $A$. 
For such an orthogonal array $A_1$ whose block graph has non-canonical cliques, we say that the block graph is $A$-\emph{extremal} if the block graphs of all orthogonal arrays of type $OA(m-1,n)$ obtained as subsets of rows of $A$ have the strict-EKR property. In this sense, extremal Peisert-type graphs considered in this paper are $\AG(2,q)$-extremal.  

In this section we discuss some implications to the study of extremal block graphs of orthogonal arrays with non-canonical cliques. In \cite[Corollary 5.5.3]{GM15}, it was shown that if $OA(m, n)$ is an orthogonal array with $n > (m - 1)^2$, then
the only cliques of size $n$ in the corresponding block graph are canonical cliques. Moreover, some non-canonical cliques in the block graphs of orthogonal arrays of type $(m,(m-1)^2)$ (the smallest case when non-canonical cliques are possible according to the bound) were discussed in \cite[Section 5.5, p. 98]{GM15}. In particular, it was noted that ``In general, if the array $OA(m, n)$ has a subarray with $m$ rows that is an $OA(m,m - 1)$, then the columns of this subarray form a clique in $X_{OA(m,n)}$
of size $(m - 1)^2$.'' Thus, the following problem is of interest.
\begin{problem}[{\cite[Problem 16.4.2]{GM15}}]\label{ProblemOA}
Assume that $OA(m, (m - 1)^2)$ is an orthogonal array and
its block graph has non-canonical cliques of size $(m - 1)^2$. Do
these non-canonical cliques form subarrays that are isomorphic to orthogonal
arrays with entries from $\{1, \ldots, m - 1\}$?
\end{problem}

We have shown that the graphs $X_q$ are the only Peisert-type graphs that correspond to the equality case in the bound above and have non-canonical cliques. In view of Problem \ref{ProblemOA}, it is interesting to ask whether the non-canonical cliques in $X_q$ correspond to a subarray $OA(\sqrt{q}+1,\sqrt{q})$ in the orthogonal array $OA(\sqrt{q}+1,q)$ whose block graph is $X_q$. Indeed, the answer is positive. 

\begin{lemma}\label{InteresectionCanonicalNoncanonical}
Let $C_1$ be a canonical clique in $X_q$, and $C_2$ be a maximum non-canonical clique in $X_q$ such that $C_1 \cap C_2 \ne \emptyset$. Then $|C_1 \cap C_2| = \sqrt{q}$.
\end{lemma}
\begin{proof}
Let $\gamma \in C_1 \cap C_2$. Then $-\gamma+C_1$ is a canonical clique containing $0$, and $-\gamma+C_2$ a maximum non-canonical clique containing $0$. In view of Proposition \ref{prop:NonCanonicalCliques} and the structure of non-canonical cliques containing 0, we conclude that $|(-\gamma+C_1) \cap (-\gamma+C_2)| = \sqrt{q}$, which implies $|C_1 \cap C_2| = \sqrt{q}$. 
\end{proof}

Let $\Pi$ be a finite projective (respectively, an affine) plane of order $n$ and $\Pi_0$ a projective (respectively, an affine) subplane of $\Pi$ of order $n_0$ different from $\Pi$; then $n_0 \le \sqrt{n}$. If $n_0 = \sqrt{n}$, then $\Pi_0$ is called a \emph{Baer subplane} of $\Pi$. Thus, Baer subplanes are the ``biggest'' possible proper subplanes of finite planes (see \cite{C72}).

\begin{proposition}\label{prop:Baer}
The non-canonical cliques in $X_q$ correspond to orthogonal subarrays $OA(\sqrt{q}+1,\sqrt{q})$, which are Baer subplanes in $\AG(2,q)$.
\end{proposition}
\begin{proof}
Let $r = \sqrt{q}$.
The canonical cliques in $X_q$ are naturally divided into $r+1$ parallel classes of size $q$.

In view of Proposition \ref{prop:NonCanonicalCliques}, all non-canonical cliques in $X_q$ are equivalent under the automorphism group of $X_q$. Thus, it suffices to show that the non-canonical clique $C_q$ defined in Section~\ref{sec:eigenfunction} has a subarray structure. For this, in view of Lemma \ref{InteresectionCanonicalNoncanonical}, it suffices to show that for every parallel class of canonical cliques, there are exactly $r$ lines intersecting $C_q$.

Let $C$ be a canonical clique containing $0$. Put $C' = C_q \cap C$. The result follows from the realization of Peisert-type graphs as point-line incidence orthogonal array (see, for example, \cite{AGLY22}) and the fact that $(\gamma_i+C)\cap C_q= \gamma_i + C'$.
\end{proof}

The main implication of Proposition~\ref{prop:Baer} is that the block graphs of orthogonal arrays obtained from the affine planes $\AG(2,q)$ do not give a negative answer for Problem \ref{ProblemOA}. The following problem naturally arises.
\begin{problem}\label{OtherPlanesProblem1}
Let $A$ be an affine plane of order $q$ such that $q$ is the square of a prime power and $A$ is not isomorphic to the affine plane $\AG(2,q)$. Do the block graphs of the orthogonal arrays obtained from $\sqrt{q}+1$ parallel classes of $A$ have non-canonical cliques without a subarray structure?
\end{problem}
In general, the following problem is of interest.
\begin{problem}\label{OtherPlanesProblem2}
Let $A$ be an affine plane of order $q$ such that $A$ is not isomorphic to an affine plane $\AG(2,q)$. What are $A$-extremal block graphs of orthogonal arrays?    
\end{problem}
Note that Problems \ref{OtherPlanesProblem1} and \ref{OtherPlanesProblem2} are special cases of the most general problem (see \cite[Section 16.4]{GM15}) of determination all the maximum cliques in the
block graph of an orthogonal array (we focus on the orthogonal arrays that can be extended to a complete orthogonal array; a criterion of extendibility is given in \cite[Theorem 10.4.5]{GR01}).

\section*{Acknowledgments}
The authors thank Andries E. Brouwer and Akihiro Munemasa for helpful discussions. The authors are also grateful to the referees for valuable corrections and suggestions.

\appendix

\section{Computational results and a further problem}\label{sec:computation}
In this section, we provide computational results on the total number of pairwise non-isomorphic Peisert-type graphs of order $q^2$ with or without the strict-EKR property for small values of $q$.

It follows from Corollary \ref{cor:3iso} that, for any prime power $q$, all Peisert-type graphs of type $(3,q)$ are isomorphic.
Since the complement of a Peisert-type graph of type $(m,q)$ is a Peisert-type graph of type $(q+1-m,q)$, we conclude that the number of pairwise non-isomorphic Peisert-type graphs of type $(m,q)$ is equal to the number of pairwise non-isomorphic Peisert-type graphs of type $(q+1-m,q)$.
Whenever $q \leq 5$, there exists a unique Peisert-type graph of type $(m,q)$ for any admissible value of $m$, that is, $1 \le m \le q$. Given a prime power $q$, let $e_q$ denote the smallest value of $m$ such that a Peisert-type graph without the strict-EKR property and of type $(m,q)$ exists. We call this value of $e_q$ \emph{extremal}. In this paper, we have determined the value of $e_q$ for any prime power $q$. 

Let us describe an algorithm for the enumeration of Peisert-type graphs of small order. Input: a prime power $q \ge 5$ and a number $3 \le m \le q-2$; output: the number of (pairwise non-isomorphic) Peisert-type graphs of type $(m,q)$ with strict-EKR property and the number of (pairwise non-isomorphic) Peisert-type graphs of type $(m,q)$ without strict-EKR property. The affine plane $\operatorname{AG}(2,q)$ has exactly $q+1$ directions; let them be indexed by the integers $1,2,\ldots, q+1$. 
Without loss of generality, we may assume that the directions with indices $q-1,q,q+1$ are in the direction set determined by the graph. In other words, to enumerate all Peisert-type graphs of type $(m,q)$, we enumerate all Peisert-type graphs of type $(m-3,q)$, with direction set not having the directions $q-1,q,q+1$ and then add the directions $q-1,q,q+1$ to the direction set to get a complete list of Peisert-type graphs of type $(m,q)$. Thus, we need to enumerate all $(m-3)$-subsets of the set $\{1,2,\ldots,q-2\}$, and we do this using the algorithm from \cite[p. 111]{B10}.
For a new Peisert-type graph $X$ of type $(m,q)$ (that is, a graph not isomorphic to every graph from a current list of pairwise non-isomorphic Peisert-type graphs), we consider the subgraph $X_0$ induced by the neighborhood of zero vertex and enumerate all its maximum cliques of size $q-1$. If $X_0$ has exactly $m$ cliques of size $q-1$, we conclude that $X$ has the strict-EKR property; otherwise, we conclude that $X$ does not have the strict-EKR property.

As an illustration, the following MAGMA \cite{BCP97} code implements the algorithm in the case $q = 17$ and $m = 6$.

\begin{lstlisting}
NextSubset:=function(CurSubset,n,r)
    if CurSubset eq [-1] then
        return [1..r];
    end if;
    X:=CurSubset;
    for k := r to 1 by -1 do
        Exclude(~X,CurSubset[k]);
        if CurSubset[k]+1 le n and CurSubset[k]+1 notin CurSubset then     
            return X cat [CurSubset[k]+1..CurSubset[k]+r-k+1];        
        end if;
    end for;    
end function;
q:=17;
m:=6;

Fq:=FiniteField(q);
FqSet:=Set(Fq);
Sq:={@i^2 : i in Fq@};
Nq:= SetToIndexedSet(Set(Fq) diff Sq);
d:=Nq[1];
R<x>:=PolynomialRing(Fq);
f := elt< R | -d, 0, 1 >;
flagIr:=IsIrreducible(f);
if flagIr then
    Fqs:={@elt<R|x,y> : x,y in Fq@};
    e0:=elt<R|1,0>;    
    Slopes:={@ e0 @};
    for el in Fq do
        Slopes:=Slopes join {elt<R|el,1>};    
    end for;            
    GrList:=[];
    cntStrictEKR:=0;
    cntWithoutStrictEKR:=0;    
    X:=[-1];
    cntSubsets:=0;
    endBound:=Binomial(q-2,m-3);
    while X ne [q-2-(m - 3)+1..q-2] do
        X:=NextSubset(X,q-2,m-3);
        cntSubsets+:=1;
        if cntSubsets mod 1000 eq 1 then
            printf"cntSubsets = %o to %o\n", cntSubsets, endBound; 
        end if;
        S:={@@};
        ss1:={q-1,q,q+1} join SequenceToSet(X);
        for ind in ss1 do
            slope:=Slopes[ind];
            S:=S join {c*slope : c in Fq | c ne 0};
        end for;
        sz:=q^2;
        A:=ZeroMatrix(IntegerRing(),sz,sz);
        for i in [1..sz-1] do
            for j in [i+1..sz] do
                if Fqs[i] - Fqs[j] in S then
                    A[i,j]:=1;
                    A[j,i]:=1;
                end if;
            end for;
        end for;
        Gr:=Graph<sz|A>;
        flagIsNewGraph:=true;
        for ind in [1..#GrList] do
            if IsIsomorphic(GrList[ind], Gr) then
                flagIsNewGraph:=false;
                break;
            end if;
        end for;
        if flagIsNewGraph then 
            GrList[#GrList+1]:=Gr;          
            flagIsStrictEKR:=true;            
            sz1:=#S;
            B:=ZeroMatrix(IntegerRing(),sz1,sz1);
            for i in [1..sz1-1] do 
                for j in [i+1..sz1] do 
                    if S[i] - S[j] in S then
                        B[i,j]:=1;
                        B[j,i]:=1;
                    end if;
                end for;
            end for;
            Gr1:=Graph<sz1|B>;
            CC:=AllCliques(Gr1,q-1 : Limit:=m+1);
            if #CC eq m+1 then
                flagIsStrictEKR:=false;
            end if;
            if flagIsStrictEKR then
                cntStrictEKR+:=1;
            else
                cntWithoutStrictEKR+:=1;
            end if;
            printf"GrNum = %o cntStrictEKR = %o cntWithout = %o\n", 
            #GrList, cntStrictEKR, cntWithoutStrictEKR; 
        end if;  
    end while;
    printf"cntStrictEKR = %o cntWithout = %o\n", 
    cntStrictEKR, cntWithoutStrictEKR;
end if;
\end{lstlisting}

The following tables, for small values of $q$, give computational information on how many Peisert-type graphs with/without the strict-EKR property there exist. In particular, the value of $e_q$ and the uniqueness of an extremal graph can be seen from the tables, and this agrees with the main results obtained in this paper. The most difficult case to compute was $q = 27$ and $m = \frac{q+1}{2} = 14$, and it took about six weeks on an ordinary laptop.

$$q = 4: \,\,\,\,\,\,\,\,\,\,\,\,\,\,\,\,\,\,\,\,\,\,\,\,\,\,\,\,\,\,\,\,\,\,\,\,\,\,\,q = 5:$$
\begin{center}
~~
\begin{tabular}{|c|c|c|c|}
  \hline
  $m$ & 3 & 4\\
  \hline
  \text{\#Graphs} & 1 & 1 \\
  \hline
  \text{strict-EKR} & - & - \\
  \hline
  \text{without} & 1 & 1 \\
  \hline
\end{tabular}\,\,\,\,\,\,
\begin{tabular}{|c|c|c|c|}
  \hline
  $m$ & 3 & 4 & 5\\
  \hline
  \text{\#Graphs} & 1 & 1 & 1\\
  \hline
  \text{strict-EKR} & 1 & - & -\\
  \hline
  \text{without} & - & 1 & 1\\
  \hline
\end{tabular}
\end{center}

$$q = 7:\,\,\,\,\,\,\,\,\,\,\,\,\,\,\,\,\,\,\,\,\,\,\,\,\,\,\,\,\,\,\,\,\,\,\,\,\,\,\,\,\,\,\,\,\,\,\,\,\,\,\,\,\,q=8:$$
\begin{center}
\begin{tabular}{|c|c|c|c|c|}
  \hline
  $m$ & 3 & 4 & 5 & 6\\
  \hline
  \text{\#Graphs} & 1 & 2 & 1 & 1\\
  \hline
  \text{strict-EKR} & 1 & 2 & - & -\\
  \hline
  \text{without} & - & - & 1 & 1\\
  \hline
\end{tabular}
\,\,\,\,\,\,
\begin{tabular}{|c|c|c|c|c|c|}
  \hline
  $m$ & 3 & 4 & 5 & 6\\
  \hline
  \text{\#Graphs} & 1 & 1 & 1 & 1\\
  \hline
  \text{strict-EKR} & 1 & 1 & - & -\\
  \hline
  \text{without} & - & - & 1 & 1 \\
  \hline
\end{tabular}
\end{center}

$$q=9: \,\,\,\,\,\,\,\,\,\,\,\,\,\,\,\,\,\,\,\,\,\,\,\,\,\,\,\,\,\,\,\,\,\,\,\,\,\,\,\,\,\,\,\,\,\,\,\,\,\,\,\,\,\,\,\,\,\,\,\,\,\,\,\,\,\,\,\,\,\,\,q=11:$$
\begin{center}
\begin{tabular}{|c|c|c|c|c|c|c|}
  \hline
  $m$ & 3 & 4 & 5 & 6 & 7\\
  \hline
  \text{\#Graphs} & 1 & 2 & 2 & 2 & 1\\
  \hline
  \text{strict-EKR} & 1 & 1 & 1 & - & -\\
  \hline
  \text{without} & - & 1 & 1 & 2 & 1\\
  \hline
\end{tabular} 
\,\,\,\,\,\,
\begin{tabular}{|c|c|c|c|c|c|c|c|c|}
  \hline
  $m$ & 3 & 4 & 5 & 6 & 7 & 8 & 9\\
  \hline
  \text{\#Graphs} & 1 & 2 & 2 & 4 & 2 & 2 & 1\\
  \hline
  \text{strict-EKR} & 1 & 2 & 2 & 4 & 1 & 1 & -\\
  \hline
  \text{without} & - & - & - & - & 1 & 1 & 1\\
  \hline
\end{tabular} 
\end{center}

$$q=13:$$
\begin{center}
\begin{tabular}{|c|c|c|c|c|c|c|c|c|c|c|}
  \hline
  $m$ & 3 & 4 & 5 & 6 & 7 & 8 & 9 & 10 & 11\\
  \hline
  \text{\#Graphs} & 1 & 3 & 3 & 5 & 5 & 5 & 3 & 3 & 1\\
  \hline
  \text{strict-EKR} & 1 & 3 & 3 & 5 & 5 & 4 & 2 & - & -\\
  \hline
  \text{without} & - & - & - & - & - & 1 & 1 & 3 & 1\\
  \hline
\end{tabular}

$$q=16:$$
\begin{tabular}{|c|c|c|c|c|c|c|c|c|c|c|c|c|c|}
  \hline
  $m$ & 3 & 4 & 5 & 6 & 7 & 8 & 9 & 10 & 11 & 12 & 13 & 14\\
  \hline
  \text{\#Graphs} & 1 & 2 & 3 & 4 & 5 & 6 & 6 & 5 & 4 & 3 & 2 & 1\\
  \hline
  \text{strict-EKR} & 1 & 2 & 2 & 3 & 3 & 3 & 1 & - & - & - & - & - \\
  \hline
  \text{without} & - & - & 1 & 1 & 2 & 3 & 5 & 5 & 4 & 3 & 2 & 1\\
  \hline
\end{tabular} 
\end{center}

$$q=17:$$
\begin{center}
\begin{tabular}{|c|c|c|c|c|c|c|c|c|c|c|c|c|c|c|}
  \hline
  $m$ & 3 & 4 & 5 & 6 & 7 & 8 & 9 & 10 & 11 & 12 & 13 & 14 & 15\\
  \hline
  \text{\#Graphs} & 1 & 3 & 4 & 10 & 10 & 17 & 17 & 17 & 10 & 10 & 4 & 3 & 1\\
  \hline
  \text{strict-EKR} & 1 & 3 & 4 & 10 & 10 & 17 & 17 & 16 & 9 & 6 & 1 & - & -\\
  \hline
  \text{without} & - & - & - & - & - & - & - & 1 & 1 & 4 & 3 & 3 & 1\\
  \hline
\end{tabular} 
\end{center}

$$q=19:$$
\begin{center}
\begin{tabular}{|c|c|c|c|c|c|c|c|c|c|c|c|c|c|c|c|c|}
  \hline
  $m$ & 3 & 4 & 5 & 6 & 7 & 8 & 9 & 10 & 11 & 12 & 13 & 14 & 15 & 16 & 17 \\
  \hline
  \text{\#Graphs} & 1 & 4 & 5 & 13 & 18 & 31 & 33 & 44 & 33 & 31 & 18 & 13 & 5 & 4 & 1\\
  \hline
  \text{strict-EKR} & 1 & 4 & 5 & 13 & 18 & 31 & 33 & 44 & 32 & 30 & 14 & 5 & - & - & -\\
  \hline
  \text{without} & - & - & - & - & - & - & - & - & 1 & 1 & 4 & 8 & 5 & 4 & 1\\
  \hline
\end{tabular}
\end{center}

$$q=23:$$
\begin{center}
\begin{tabular}{|c|c|c|c|c|c|c|c|c|c|c|c|}
  \hline
  $m$ & 3 & 4 & 5 & 6 & 7 & 8 & 9 & 10 & 11 & 12 & 13\\
  \hline
  \text{\#Graphs} & 1 & 4 & 6 & 22 & 36 & 83 & 125 & 196 & 227 & 268 & 227\\
  \hline
  \text{strict-EKR} & 1 & 4 & 6 & 22 & 36 & 83 & 125 & 196 & 227 & 268 & 226\\
  \hline
  \text{without} & - & - & - & - & - & - & - & - & - & - & 1\\
  \hline
\end{tabular}
\\
\begin{tabular}{|c|c|c|c|c|c|c|c|c|}
  \hline
  $m$ & 14 & 15 & 16 & 17 & 18 & 19 & 20 & 21\\
  \hline
  \text{\#Graphs} & 196 & 125 & 83 & 36 & 22 & 6 & 4 & 1 \\
  \hline
  \text{strict-EKR} & 195 & 120 & 73 & 19 & 5 & - & - & - \\
  \hline
  \text{without} & 1 & 5 & 10 & 17 & 17 & 6 & 4 & 1 \\
  \hline
\end{tabular} 
\end{center}

$$q=25:$$
\begin{center}
\begin{tabular}{|c|c|c|c|c|c|c|c|c|c|c|c|}
  \hline
  $m$ & 3 & 4 & 5 & 6 & 7 & 8 & 9 & 10 & 11 & 12 & 13\\
  \hline
  \text{\#Graphs} & 1 & 4 & 7 & 19 & 34 & 79 & 132 & 223 & 293 & 379 & 391\\
  \hline
  \text{strict-EKR} & 1 & 4 & 7 & 18 & 33 & 75 & 121 & 185 & 208 & 198 & 108\\
  \hline
  \text{without} & - & - & - & 1 & 1 & 4 & 11 & 38 & 85 & 181 & 283\\
  \hline
\end{tabular}
\\
\begin{tabular}{|c|c|c|c|c|c|c|c|c|c|c|}
  \hline
  $m$ & 14 & 15 & 16 & 17 & 18 & 19 & 20 & 21 & 22 & 23\\
  \hline
  \text{\#Graphs} & 379 & 293 & 223 & 132 & 79 & 34 & 19 & 7 & 4 & 1\\
  \hline
  \text{strict-EKR} & 34 & 3 & 1 & - & - & - & - & - & - & -\\
  \hline
  \text{without} & 345 & 290 & 222 & 132 & 79 & 34 & 19 & 7 & 4 & 1\\
  \hline
\end{tabular} 
\end{center}

$$q=27:$$
\begin{center}
\begin{tabular}{|c|c|c|c|c|c|c|c|c|c|c|c|c|}
  \hline
  $m$ & 3 & 4 & 5 & 6 & 7 & 8 & 9 & 10 & 11 & 12 & 13 & 14\\
  \hline
  \text{\#Graphs} & 1 & 3 & 4 & 14 & 29 & 72 & 134 & 257 & 390 & 565 & 670 & 738\\
  \hline
  \text{strict-EKR} & 1 & 3 & 4 & 14 & 29 & 72 & 134 & 256 & 389 & 560 & 651 & 670\\
  \hline
  \text{without} & - & - & - & - & - & - & - & 1 & 1 & 5 & 19 & 68\\
  \hline
\end{tabular}
\\
\begin{tabular}{|c|c|c|c|c|c|c|c|c|c|c|c|}
  \hline
  $m$ & 15 & 16 & 17 & 18 & 19 & 20 & 21 & 22 & 23 & 24 & 25\\
  \hline
  \text{\#Graphs} & 670 & 565 & 390 & 257 & 134 & 72 & 29 & 14 & 4 & 3 & 1\\
  \hline
  \text{strict-EKR} & 513 & 265 & 57 & 4 & - & - & - & - & - & - & -\\
  \hline
  \text{without} & 157 & 300 & 333 & 253 & 134 & 72 & 29 & 14 & 4 & 3 & 1\\
  \hline
\end{tabular} 
\end{center}

Given a prime power $q$, let $E_q$ denote the largest value of $m$ such that a Peisert-type graph with the strict-EKR property and of type $(m,q)$ exists. We also call this value of $E_q$ and Peisert-type graphs with the strict-EKR property, of type $(E_q,q)$ \emph{extremal}. It follows from the above tables that $E_4 = 2, E_5 = 3, E_7 = 4, E_8 = 4, E_9 = 5, E_{11} = 8, E_{13} = 9, E_{16} = 9, E_{17} = 13, E_{19} = 14, E_{23} = 18$, $E_{25} = 16$ and $E_{27} = 18$.
\begin{problem}
For each prime power $q$, determine the value of $E_q$ and the number of pairwise non-isomorphic extremal Peisert-type graphs with the strict-EKR property and of type $(E_q,q)$.
\end{problem}

\end{document}